\documentclass[12pt]{article}
\usepackage{amssymb}
\usepackage{amsmath,bm,natbib,graphics,mathrsfs,graphicx,epsfig,subfigure,placeins,indentfirst,setspace,mhchem,enumerate,enumitem,caption,varioref,booktabs,tabularx,color,epstopdf,multirow}
\setcitestyle{authoryear,open={(},close={)}}
\usepackage{amsthm}
\usepackage{lscape}
\makeatletter \oddsidemargin  -.1in \evensidemargin -.1in
\begin{document}
	\newcommand{\bea}{\begin{eqnarray}}
		\newcommand{\eea}{\end{eqnarray}}
	\newcommand{\nn}{\nonumber}
	\newcommand{\bee}{\begin{eqnarray*}}
		\newcommand{\eee}{\end{eqnarray*}}
	\newcommand{\lb}{\label}
	\newcommand{\nii}{\noindent}
	\newcommand{\ii}{\indent}
	\newtheorem{theorem}{Theorem}[section]
	\newtheorem{example}{Example}[section]
	\newtheorem{corollary}{Corollary}[section]
	\newtheorem{definition}{Definition}[section]
	\newtheorem{lemma}{Lemma}[section]
	\newtheorem{remark}{Remark}[section]
	\newtheorem{proposition}{Proposition}[section]
	\numberwithin{equation}{section}
	\renewcommand{\qedsymbol}{\rule{0.7em}{0.7em}}
	\renewcommand{\theequation}{\thesection.\arabic{equation}}
	\renewcommand\bibfont{\fontsize{10}{12}\selectfont}
	\setlength{\bibsep}{0.0pt}
		\title{\bf General weighted information and relative information generating functions with properties**}
	
\author{ Shital {\bf Saha}\thanks {Email addresses: shitalmath@gmail.com,~520MA2012@nitrkl.ac.in} ~and  Suchandan {\bf  Kayal}\thanks {Email addresses (corresponding author):
		suchandan.kayal@gmail.com,~kayals@nitrkl.ac.in
		\newline**To appear on \textbf{IEEE Transactions on Information Theory}.}
	\\{\it \small Department of Mathematics, National Institute of
			Technology Rourkela, Rourkela-769008, India}}
\date{}
\maketitle
\begin{center}
\textbf{Abstract}
\end{center}
In this work, we propose two information generating functions: general weighted information and relative information generating functions, and study their properties. { It is shown that the general weighted information generating function (GWIGF) is shift-dependent and can be expressed in terms of the weighted Shannon entropy. The GWIGF of a transformed random variable has been obtained in terms of the GWIGF of a known distribution. Several bounds of the GWIGF have been proposed. We have obtained sufficient conditions under which the GWIGFs of two distributions are comparable. Further, we have established a connection between the weighted varentropy and varentropy with proposed GWIGF. An upper bound for GWIGF of the sum of two independent random variables is derived. The effect of general weighted relative information generating function (GWRIGF) for two transformed random variables under strictly monotone functions has been studied. } Further, these information generating functions are studied for escort, generalized escort and mixture distributions. {Specially, we propose weighted $\beta$-cross informational energy and establish a close connection with GWIGF for escort distribution.} The residual versions of the newly proposed generating functions are considered and several similar properties have been explored. A non-parametric estimator of the residual general weighted information generating function is proposed. A simulated data set and two real data sets are considered for the purpose of illustration. { Finally, we have compared the non-parametric approach with a parametric approach in terms of the absolute bias and mean squared error values.}
			
\section{Introduction}
In recent past, information theory has become an important field of reseach, which mainly deals with the quantification of the lack of sureness associated with a distribution. The concept of entropy was first proposed by \cite{shannon1948mathematical} in information theory. Due to this path-breaking contribution, C. Shannon is known as the father of information theory. Though entropy was first developed for discrete random variables (RVs), here we present its developments for the non-negative and absolutely continuous RVs, since in this work, we mainly focus on the information generating functions (IGFs) for continuous case. Suppose $X$ is a non-negative absolutely continuous RV with probability density function (PDF) $f(\cdot)$. The Shannon entropy for $X$ is given as  
\begin{eqnarray}\label{eq1.1}
H(X)=E(-\log f(X))=-\int_{0}^{\infty}f(x)\log f(x)dx.
\end{eqnarray}
$H(X)$ given above is also called as the differential entropy. Throughout the paper, we treat `$\log$' as natural logarithm. It is known that $H(X)$ is position-free. For $a,b>0$ $H(aX+b)=H(X)+\log a.$ It happens since (\ref{eq1.1}) does not include the value of RV. It depends on its PDF. Due to this property, the Shannon entropy is not applicable to some situations, where a biological system is supposed to produce different uncertainty measures at two different locations. Motivated by the works of \cite{belis1968quantitative} and \cite{guiacsu1971weighted}, weighted form of $H(X)$ has been proposed by \cite{di2007weighted}. Henceforth, we call it weighted entropy instead of weighted Shannon entropy. The weighted entropy for $X$ is 
\begin{eqnarray}\label{eq1.2}
H^{x}(X)=\int_{0}^{\infty}(\log f(x)) x f(x) dx,
\end{eqnarray}
where the factor $x$ is called as weight, linearly focusing the occurrence of the event, say $\{X=x\}.$ The presence of the factor $x$ inside the integral of (\ref{eq1.2}) includes more importance when  $X$ assumes larger values.  Clearly, when weight is equal to $1$, (\ref{eq1.2}) reduces to (\ref{eq1.1}). 

The most convenient tool to generate moments of a probability distribution is its moment generating function (MGF). All the raw moments of a distribution, if exists, can be obtained after successive derivatives of the MGF, say $M_{X}(s)=E(e^{sX})$ with respect to $s$, at $s=0,$ that is $E(X^{k})=\frac{d^k}{ds^k}M_{X}(s)|_{s=0}.$ Motivating with this well-known fact, \cite{golomb1966} first proposed IGF of a distribution whose first order derivative evaluated at $1$ yields negative Shannon entropy or negentropy of a distribution. The IGF of $X$ with PDF $f(\cdot)$ for $\beta\ge1$ is
\begin{eqnarray}\label{eq1.3}
I_{\beta}(X)=\int_{0}^{\infty}f^{\beta}(x)dx.
\end{eqnarray}
The reason for taking $\beta$ to be larger than $1$, has been explained by \cite{golomb1966}. Clearly, $I_{\beta}(X)|_{\beta=1}=1$  and $\frac{d}{d\beta}I_{\beta}(X)|_{\beta=1}=-H(X).$ After almost $20$ years, the relative information generating function (RIGF) was introduced by \cite{guiasu1985relative}. It is shown that its derivative with respect to $\beta$ computed at $1$ produces the well-known Kullback-Leibler (KL) divergence. Suppose $X$ and $Y$ have PDFs $f(\cdot)$ and $g(\cdot),$ respectively. The RIGF of $X$ given the reference variable $Y$ is 
\begin{eqnarray}\label{eq1.4}
R_{\beta}(X,Y)=\int_{0}^{\infty}f^{\beta}(u)g^{1-\beta}(u)du,~\beta\ge1.
\end{eqnarray}
Obviously, $R_{\beta}(X,Y)|_{\beta=1}=1$ and $\frac{d}{d\beta}R_{\beta}(X,Y)|_{\beta=1}=\int_{0}^{\infty}\log(\frac{f(u)}{g(u)}) f(u)du,$ which is dubbed as KL divergence between $X$ and $Y.$ For elaborate discussion about KL divergence, we refer the reader to \cite{kullback1951information}. \cite{guiasu1985relative} also showed that the KL $J$-divergence between $X$ and $Y$ can be derived using (\ref{eq1.4}) as follows:
\begin{eqnarray}
J(X,Y)&=&\int_{0}^{\infty}\left(\log\left(\frac{f(x)}{g(x)}\right) f(x)+\log\left(\frac{g(x)}{f(x)}\right)g(x)\right)dx\nonumber\\
&=&\frac{\partial}{\partial\beta}R_{\beta}(X,Y)|_{\beta=1}+\frac{\partial}{\partial\beta}R_{\beta}(Y,X)|_{\beta=1}.
\end{eqnarray} 

After almost $50$ years since \cite{golomb1966}'s contribution, researchers have shown their interest to further develop IGFs and study their properties for various other information measures. For example, \cite{papaioannou2007some} proposed Fisher information and divergence generating functions. \cite{kharazmi2021jensen} introduced new divergence generating function, which yields some well-known measures. Based on the survival function, cumulative and relative cumulative IGFs have been proposed by \cite{kharazmi2021cumulative}. The authors have studied various properties of these generating functions. \cite{kharazmi2023optimal} proposed Jensen-RIGF and study its connections with Jensen-Shannon entropy. Besides these works, we also refer to \cite{kharazmi2022generating},  \cite{zamani2022information}, \cite{kharazmi2023optimal}, \cite{smitha2023dynamic}, and \cite{kharazmi2023jensen}.

{ Consider a probabilistic experiment $X$ having outcomes $x_{1},\cdots,x_{n}$ with corresponding probabilities $p_1,\cdots, p_n$, respectively, such that $p_i>0$, $i=1,\cdots,n$ and $\sum_{i=1}^{n}p_i=1.$ Then, the IGF is defined as (see \cite{golomb1966})
\begin{eqnarray}\label{eq1.6*}
	I_{\beta}(X)=\sum_{i=1}^{n}p_i^{\beta},~~\beta\ge1.
\end{eqnarray}
Note that (\ref{eq1.6*}) depends only on the probabilities with which various outcomes occur. However (\ref{eq1.6*}) is not useful in many fields, dealing with random experiments where it is necessary to take into account both probabilities and some qualitative characteristic of these events. Thus, for distinguishing the outcomes $x_1,\dots,x_n$ of a goal-directed experiment according to their importance with respect to a given qualitative characteristic of the system, we will assign non-negative number $\omega_k\ge0$ to each outcome $x_k$. Here, one may choose $\omega_k$, directly proportional to the importance of the $k$th outcome. Note that $\omega_k$'s are known as the weights of the outcomes $x_k$, $k=1,\ldots,n$. This type experiment is called as a probabilistic experiment with weight. For such kind of experiments the weighted IGF is useful. It is defined as 
\begin{align}\label{eq1.7}
	I_\beta^{\omega}(p)=\sum_{i=1}^{n}\omega_ip^\beta_i,~~\beta\ge1.
\end{align}
In the subsequent section, we consider the continuous analogue of (\ref{eq1.7}). Further, we present an example to show its importance.
\begin{example}
	Consider a discrete type random variable $X\in\{1,2,3\}$ with $P(X=1)=p_1$, $P(X=2)=p_2$ and $P(X=3)=p_3$. For the sets of probabilities $\{p_1=0.6, p_2=0.1,p_3=0.3\}$ and $\{p_1=0.3, p_2=0.6,p_3=0.1\}$, the values of IGF (\cite{golomb1966}) are equal to $(0.6)^\beta+(0.3)^\beta+(0.1)^\beta$. However, if we consider weights corresponding to each outcomes, then they are not equal. For example, let $\omega_1=1,\omega_2=2,\omega_3=3$. Then, for the sets of probabilities $\{p_1=0.6, p_2=0.1,p_3=0.3\}$ and $\{p_1=0.3, p_2=0.6,p_3=0.1\}$, the weighted IGFs are respectively equal to $(0.6)^\beta+2(0.1)^\beta+3(0.3)^\beta$ and $(0.3)^\beta+2(0.6)^\beta+3(0.1)^\beta$, which are clearly different from each other for $\beta\ge1.$
\end{example}
For $\beta=2,$ the IGF in (\ref{eq1.6*}) reduces to informational energy (IE), proposed by \cite{onicescu1966theorie}. The IE has many applications in different areas. For example, \cite{flores2016informational} have shown how the IE can be applied as a correlation measure in systems of atoms and molecules.  In addition, the authors have offered a useful examination of the correlation effects among the atoms in a collection of 1864 molecules. \cite{ou2019benchmark} have demonstrated the use of highly correlated Hylleraas wave functions in the analysis of the ground state helium by the use of Onicescu's informational energy. In a similar way, the weighted IGF in (\ref{eq1.7}) reduces to the weighted informational energy introduced by \cite{pardo1986order}. }
In this paper, we focus on the extension of two of these aforementioned specific results {(due to \cite{golomb1966}) and \cite{guiasu1985relative})} into a more general result. The key contributions of this paper are as follows.

\begin{itemize}
	\item[1.] We propose general information and relative information generating functions by introducing a positive-valued measurable weight function $\omega(\cdot)$ with the $\beta$th power of the PDF $f(\cdot)$ inside the integral expression of (\ref{eq1.3}) and that with the $\beta$th power of $f(.)$ and $(1-\beta)$th power of $g(.)$ inside the integral of (\ref{eq1.4}). We shall see later that our results include most of the previous results related to IGF in (\ref{eq1.3}) and RIGF in (\ref{eq1.4}) as corollaries and some special cases. 
	
	\item[2.] Our results provide a comprehensive treatment for a large class of IGFs including generating functions proposed by \cite{golomb1966} and \cite{guiasu1985relative}. We study the proposed general IGFs for the escort, generalized escort and $(r,\gamma)$-mixture models. { Further, we propose weighted $\beta$-cross informational energy and obtain a close relation with GWIGF for escort distribution.} It is shown that the weighted varentropy can be deduced from the general weighted IGF. We also show that the KL weighted $J$-divergence is obtained using the proposed concept of general weighted RIGF. 
	 
	\item[3.] The residual lifetime describes the future life of a system given that the system has survived upto a certain age. Considering the importance of the proposed general information and relative information generating functions in such practical situations, we  study residual versions of the proposed generating functions.
	
	\item[4.] We  propose a non-parametric estimator of the residual weighted general IGF. Further, a simulated data set and two data sets related to the remission times (in months) of $128$ bladder cancer patients (see \cite{lee2003statistical}) and relief times of $20$ patients who received an analgesic (see \cite{gross1975survival}) have been considered for illustrative purposes. {Goodness of fit tests are applied to see the best fitted statistical models. Finally, a comparison study has been carried out between the non-parametric approach and a parametric approach in terms of the absolute bias and mean squared error values.}
\end{itemize}

The article is arranged as follows. A general IGF with a general weight $\omega(x)$ is proposed in Section $2$. Some properties with the effect of monotone transformations and bounds are studied. Sufficient conditions under which the general weighted IGFs of two distributions are ordered have been obtained. Section $3$ focuses on the general weighted RIGF and addresses its various properties. Here, we have found a connection between weighted KL $J$-divergence and the newly proposed general RIGF. Section $4$ studies the proposed general IGFs for escort, generalized escort and $(r,\gamma)$-mixture distributions. { In addition, we propose weighted $\beta$-cross informational energy and obtain a close relation with GWIGF for escort distribution.} The residual versions of the proposed IGFs with several properties have been studied in Section $5$. In Section $6,$ we propose a non-parametric estimator of the residual general weighted IGF. Further, three data sets have been considered to compute the bias and mean squared error of the proposed estimator. {We have also compared the non-parametric approach with a parametric approach in terms of the absolute bias and mean squared errors.} Finally, Section $7$ concludes the paper.

Throughout the text, the RVs are assumed to be non-negative and absolutely continuous. The expectations, differentiations and integrations are assumed to exist. Decreasing and increasing are used in non-strict sense.

\section{General weighted information generating function}
This section addresses various properties of the general weighted information generating function (GWIGF) of a continuous probability distribution.
\begin{definition}\label{def2.1}
	The GWIGF of an RV $X$ with PDF $f(\cdot)$ is defined as 
	\begin{eqnarray}\label{eq2.1}
		I_{\beta}^{\omega}(X)=\int_{0}^{\infty}\omega(u)f^{\beta}(u)du,~~~~\beta\ge1,
	\end{eqnarray}
provided the integral exists, where $\omega(x)\ge0$ is known as the weight or utility function.  
\end{definition}
For $\omega(x)=1$, (\ref{eq2.1}) reduces to the IGF, proposed by \cite{golomb1966}. This is the reason, we call the proposed IGF in Definition \ref{def2.1} as the GWIGF. Further, $I_{\beta}^{\omega}(X)$ is always non-negative. The $k$th order derivative of (\ref{eq2.1}) with respect to $\beta$ is
\begin{eqnarray}\label{eq2.2}
	\frac{\partial^{k}I_{\beta}^{\omega}(X)}{\partial \beta^{k}}=\int_{0}^{\infty}\omega(x)f^{\beta}(x)(\log f(x))^{k}dx,~~~~\beta\ge1.
\end{eqnarray} 

Like MGF, the GWIGF in (\ref{eq2.1}) is convex since $\frac{\partial^{2}I_{\beta}^{\omega}(X)}{\partial \beta^{2}}\ge 0.$ From (\ref{eq2.1}) and (\ref{eq2.2}), the following observations are easily made:
\begin{itemize}
	\item $I_{\beta}^{\omega}(X)|_{\beta=1}=E[\omega(X)];$
	\item $I_{\beta}^{\omega}(X)|_{\beta=2}=-2J^{\omega}(X);$ 
	\item $\frac{\partial I_{\beta}^{\omega}(X)}{\partial \beta}|_{\beta=1}=-H^{\omega}(X),$
\end{itemize}
where $J^{\omega}(X)=-\frac{1}{2}\int_{0}^{\infty}\omega(u)f^{2}(u)du$ is the weighted extropy (see \cite{gupta2023general}) and $H^{\omega}(X)$ is the weighted Shannon entropy, which is also dubbed as the weighted differential entropy.  $H^{\omega}(X)$ is studied by \cite{di2007weighted} for a special weight $\omega(x)=x.$ In above, we have noticed that the first order derivative of $I_{\beta}^{\omega}(X)$, computed at $\beta=1$, yields the negative weighted Shannon entropy  (weighted negentropy). Further, when $\beta=2,$ the GWIGF $I_{\beta}^{\omega}(X)$ reduces to $\int_{0}^{\infty}\omega(u)f^{2}(u)du,$ dubbed as the weighted informational energy. We recall that the discrete version of the informational energy measure was first introduced by \cite{onicescu1966theorie}. Later on, discrete as well as the continuous versions of the information energy function have been studied in different fields by several authors, such as \cite{pardo2003test} \cite{flores2016informational} and \cite{nielsen2022onicescu}.  In the following example, we obtain  GWIGF. Two different choices of $\omega(x)$, say $x$ and $\frac{1}{x}$ are considered for the purpose of computation. 
\begin{example}~\label{ex2.1}
	\begin{itemize}
		\item[(a)] Consider uniform distribution in $(a,b)$, denoted by $U(a,b)$. Then, 
		\begin{itemize}
			\item[(i)] for $\omega(x)=x,$ we obtain $I_{\beta}^{x}(X)=\frac{a+b}{2}\left(\frac{1}{b-a}\right)^{\beta-1}$;
			\item[(ii)] for $\omega(x)=\frac{1}{x},$ we obtain $I_{\beta}^{x^{-1}}(X)=(b-a)^{\beta}\log\left(\frac{b}{a}\right)$. 
		\end{itemize}
	\item[(b)] Let $X$ follow inverted exponential distribution with PDF $f(x)=\frac{1}{\lambda x^2}e^{-\frac{1}{\lambda x}},$ $x>0,~\lambda>0,$ denoted by $IE(\lambda).$ Then, 
	\begin{itemize}
		\item[(i)] for $\omega(x)=x,$ we get $I_{\beta}^{x}(X)=\frac{\lambda^{\beta+2}}{\beta^{2\beta+2}}\Gamma(2\beta+2)$;
		\item[(ii)] for $\omega(x)=\frac{1}{x},$ we get $I_{\beta}^{x^{-1}}(X)=\frac{\lambda^{\beta}}{\beta^{2\beta}}\Gamma(2\beta)$,
	\end{itemize}
where $\Gamma(\cdot)$ denotes the complete gamma function.
	\end{itemize}
\end{example}

\begin{remark}
	The closed-form expression of the weighted Shannon entropy of U$(a,b)$ distribution can be easily obtained from Example \ref{ex2.1}$(a)$ as follows:
	\begin{itemize}
		\item[(i)] for $\omega(x)=x$, $H^{x}(X)=-\frac{\partial I_{\beta}^{x}(X)}{\partial \beta}|_{\beta=1}=\frac{(a+b)}{2}\log(b-a);$
		\item[(ii)] for $\omega(x)=1/x$, $H^{x^{-1}}(X)=-\frac{\partial I_{\beta}^{x^{-1}}(X)}{\partial \beta}|_{\beta=1}=\log(\frac{b}{a})\log(b-a)\left({b-a}\right)^{\beta};$
	\end{itemize}
and that for inverted exponential distribution can be obtained from Example \ref{ex2.1}$(b)$ as:
\begin{itemize}
	\item[(i)] for $\omega(x)=x$, $H^{x}(X)=-\frac{\partial I_{\beta}^{x}(X)}{\partial \beta}|_{\beta=1}=\frac{\lambda^\beta}{\beta^{2\beta}}[2\psi^*(2\beta)+(\log(\frac{\lambda}{\beta^2})-2)\Gamma(2\beta)];$
	\item[(ii)] for $\omega(x)=1/x$, $H^{x^{-1}}(X)=-\frac{\partial I_{\beta}^{x^{-1}}(X)}{\partial \beta}|_{\beta=1}=\frac{\lambda^{\beta+2}}{\beta^{2\beta+2}}[2\psi^*(2\beta+2)+(\log(\frac{\lambda}{\beta^2})-\frac{2}{\beta}-2)\Gamma(2\beta+2)],$
\end{itemize}
where $\psi^*$ is a digamma function. 
\end{remark}

To see the monotonicity property of GWIGF with respect to the parameter $\beta$, the  plots of GWIGF for uniform and inverted exponential distributions are presented in Figure $1.$ From Figures $1(a)$ and $1(b)$, we observe that the GWIGF of uniform distribution is decreasing with respect to $\beta$ for both $\omega(x)=x$ and $\frac{1}{x}.$ On the other hand, for the case of inverted exponential distribution, the GWIGF is not monotone in $\beta$ (see Figures $1(c)$ and $1(d)$). Thus, we conclude that the GWIGF is not monotone in general.

	\begin{figure}[h]
	\begin{center}
		\subfigure[]{\label{c1}\includegraphics[height=1.9in]{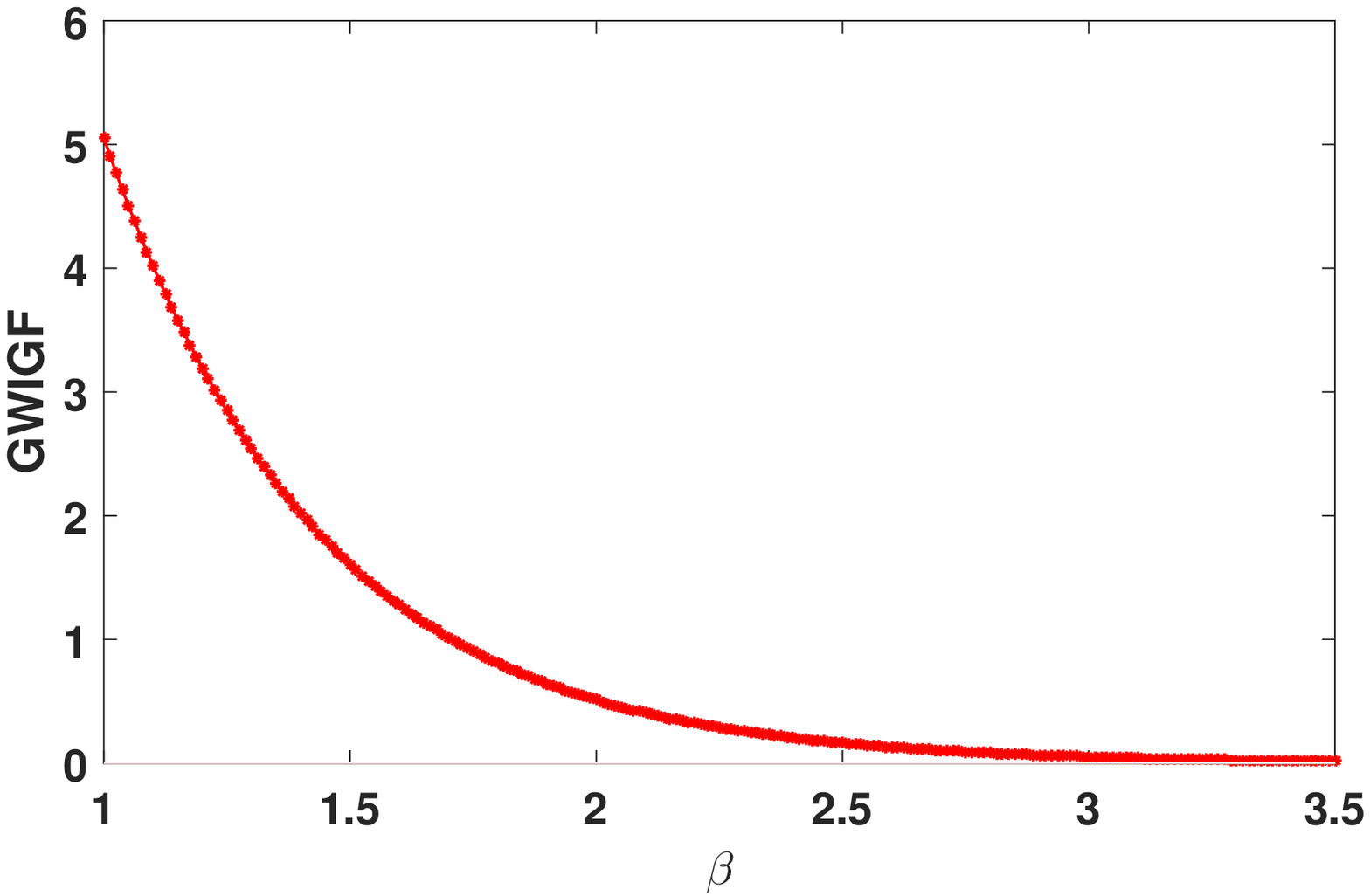}}
		\subfigure[]{\label{c1}\includegraphics[height=1.9in]{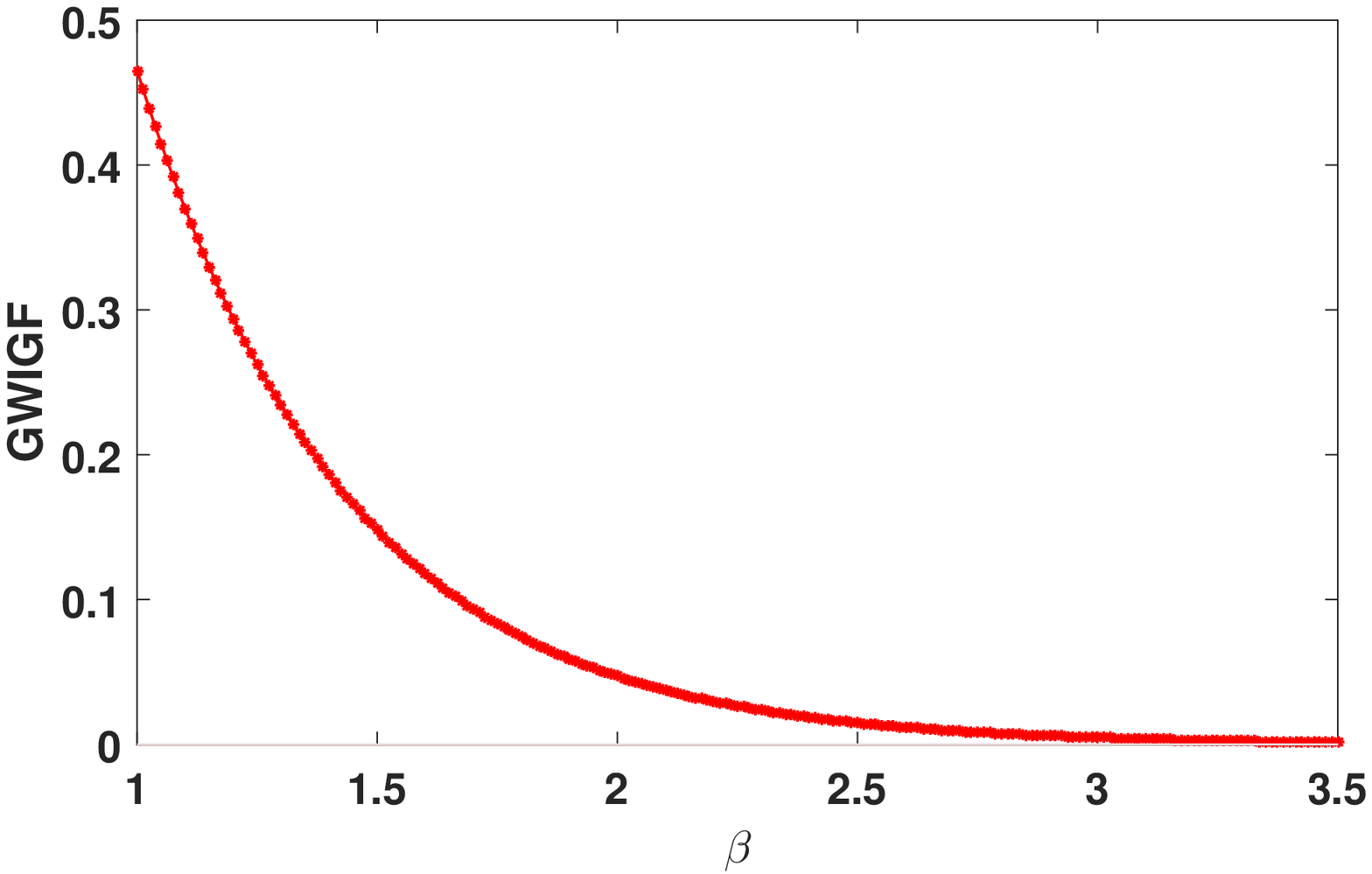}}
		\subfigure[]{\label{c1}\includegraphics[height=1.9in]{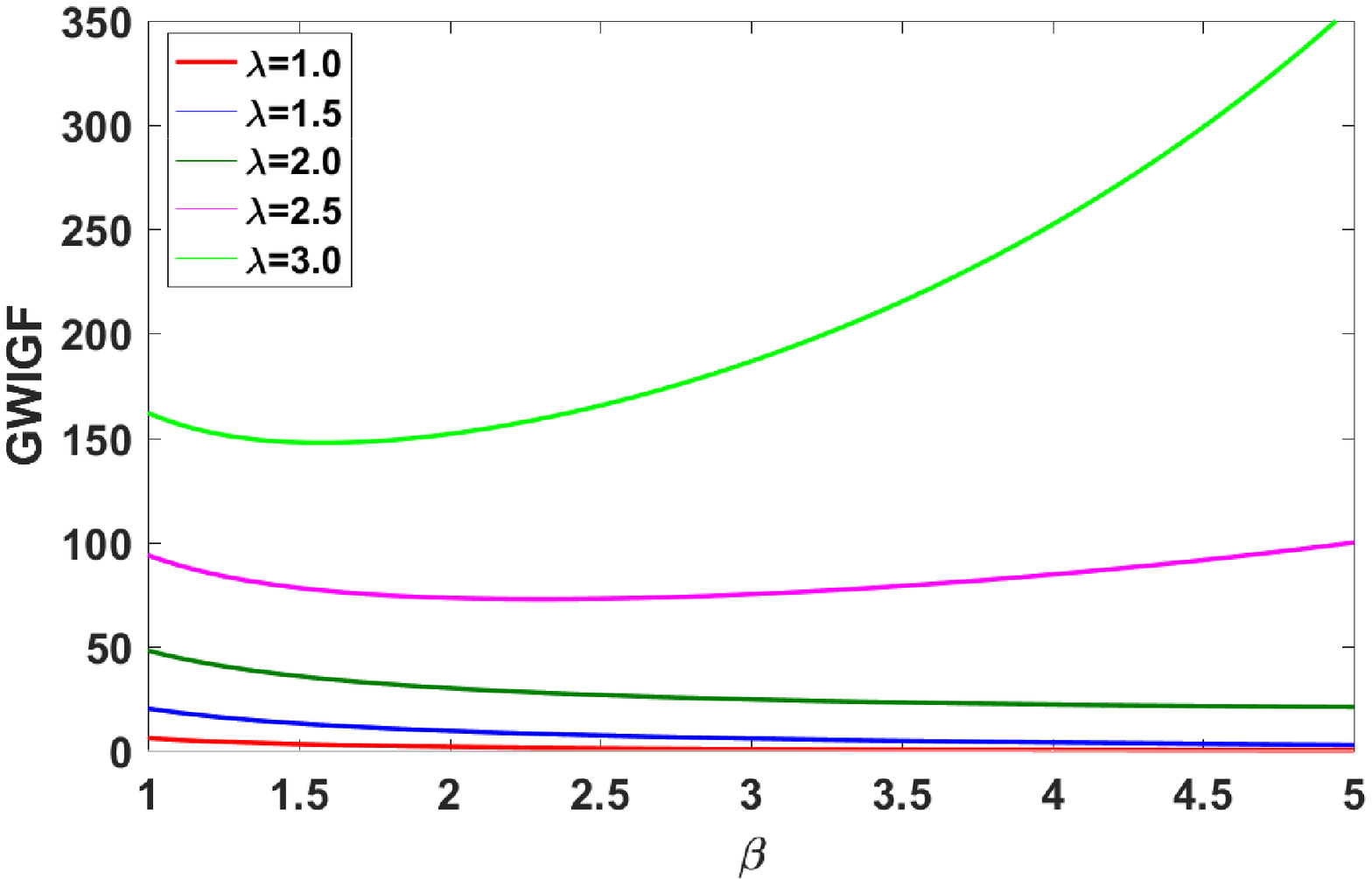}}
		\subfigure[]{\label{c1}\includegraphics[height=1.9in]{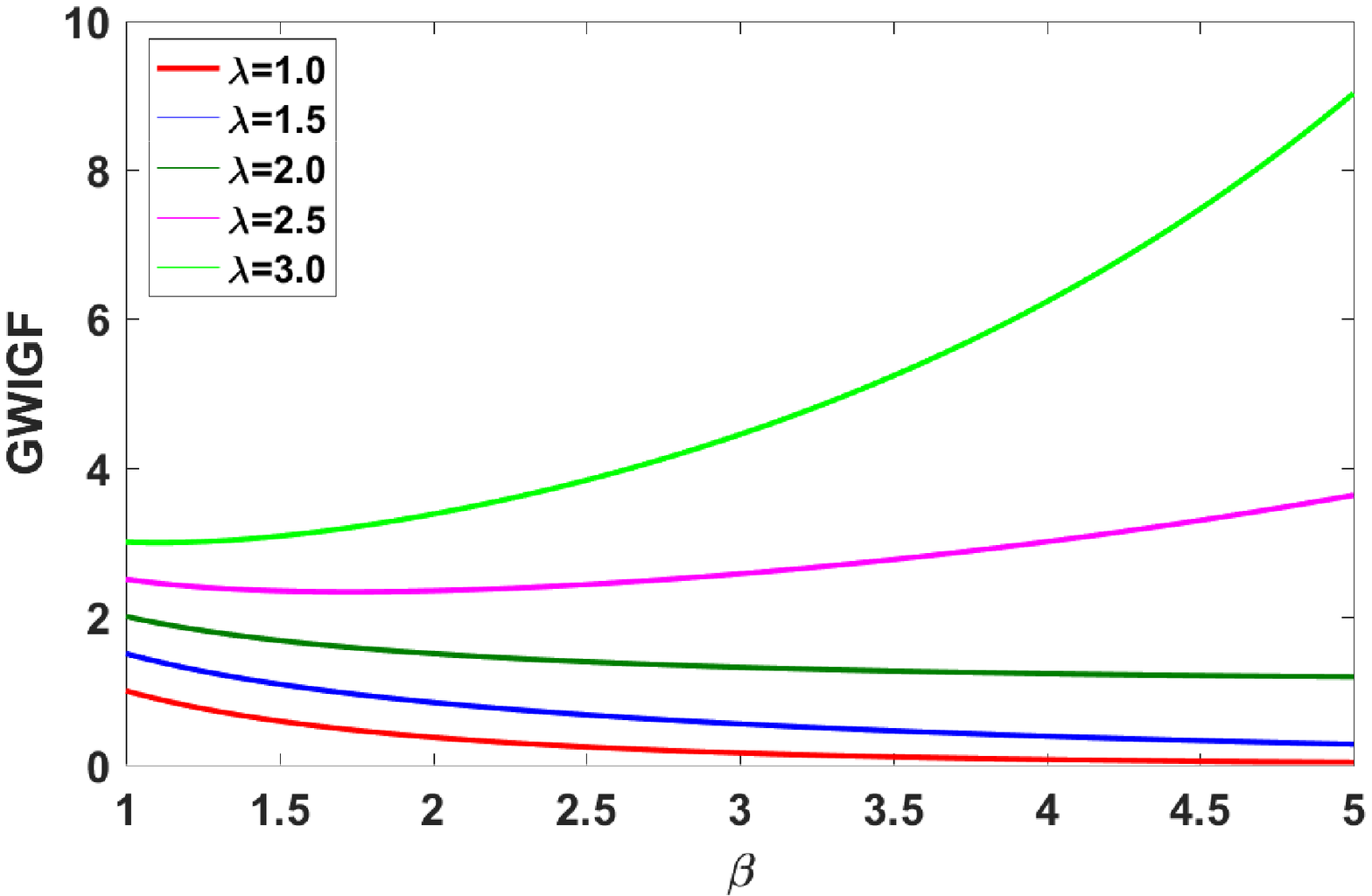}}
		
		\caption{Plot of GWIGF for $U(0,10)$ considered in Example \ref{ex2.1}$(i)$, for $(a)$ $\omega(x)=x$ and $(b)$ $\omega(x)=\frac{1}{x}$; and  of $IE(\lambda)$  as in Example \ref{ex2.1}$(ii),$ for $(c)$ $\omega(x)=x$ and $(d)$ $\omega(x)=\frac{1}{x}$.}	
	\end{center}
\end{figure}

Let $Y=aX+b$, $a,b>0$ be the affine transformations, where $X$ has PDF $f(\cdot).$  Then, the PDF of $Y$ is $g(x)=\frac{1}{a}f(\frac{x-b}{a}),~x>b.$ Using this, we obtain a relationship for the IGFs of $Y$ and $X$, given by
\begin{eqnarray}\label{eq2.3}
I_{\beta}(Y)=\frac{1}{a^{\beta-1}}I_{\beta}(X),
\end{eqnarray}
which implies that the IGF is shift-independent (independent of $b$). However, such property is not satisfied by GWIGF, which is presented below.  
\begin{proposition}\label{prop2.1}
	 Let $X$ be an RV with PDF $f(\cdot).$   Then, for $a,b>0$ 
	 \begin{eqnarray}\label{eq2.4}
	 I_{\beta}^{\omega}(Y=aX+b)=\frac{1}{a^{\beta-1}}\int_{0}^{\infty}\omega(ax+b)f^{\beta}(x)dx.
	 \end{eqnarray} 
\end{proposition}

\begin{proof}
	The proof follows using simple arguments. Thus, it is not presented.  
\end{proof}

{Next, we present a remark to show the importance of the proposed GWIGF over IGF. }

\begin{remark}
	 {From (\ref{eq2.3}), we observe that the previously proposed IGF is shift-independent. However, there are some applied contexts, such as reliability or mathematical neurobiology, where it is required to deal with some shift-dependent information measures. Note that knowing a component fails to operate, or a neuron to release spikes in a specified time duration, yields a relevantly different information from the case when such an event occurs in a different equally wide interval. In this situation, using the proposed GWIGF, one will get better result than by using the IGF proposed by \cite{golomb1966} since the GWIGF is shift-dependent.} This property reveals the importance of studying the proposed generating function in Definition \ref{def2.1}.
\end{remark}

For studying (\ref{eq2.4}) further, we consider $\omega(x)=x.$ In this case, we obtain
\begin{eqnarray}\label{eq2.5}
I_{\beta}^{\omega}(Y)=\frac{1}{a^{\beta-2}}\int_{0}^{\infty}xf^{\beta}(x)dx+\frac{b}{a^{\beta-1}}\int_{0}^{\infty}f^{\beta}(x)dx=\frac{1}{a^{\beta-2}}I_{\beta}^{x}(X)+\frac{b}{a^{\beta-1}}I_{\beta}(X),
\end{eqnarray} 
where $I_{\beta}^{x}(X)$ and $I(X)$ are respectively known as the GWIGF with weight $x$ and IGF, due to \cite{golomb1966}. {Next, we consider two distributions having equal value of  IGF but different values of GWIGF with corresponding weight $\omega(x)=x.$}

\begin{example}
Suppose $X_1$ and $X_2$ have respective PDFs
\begin{align}
f_{1}(u)=\begin{cases}
2u,~~\text{if} ~u\in(0,1),\\
0,~~\text{otherwise},
\end{cases} \mbox{and}~~
f_{2}(u)=\begin{cases}
2(1-u),~~\text{if} ~u\in(0,1),\\
0,~~\text{otherwise}.
\end{cases}
\end{align}	  
Here, we obtain 
$$I_{\beta}(X_1)=I_{\beta}(X_2)=\frac{2^{\beta}}{\beta+1},~~\beta\ge1.$$
But, 
$$I_{\beta}^{x}(X_1)=\frac{2^{\beta}}{\beta+2}~~\mbox{and}~~I_{\beta}^{x}(X_2)=2^{\beta}\mathbb{B}(2,\beta+1),~~\beta\ge1,$$
where $\mathbb{B}(\cdot,\cdot)$ denotes complete beta function. 
\end{example}

The Shannon entropy of order $k$ (a positive integer) of an RV $X$ is given by (see \cite{kharazmi2021jensen}) 
\begin{eqnarray}
H_{k}(X)=\int_{0}^{\infty}(-\log f(x))^{k} f(x)dx.
\end{eqnarray}
{In general, it is always of interest to see if there is any relationship between a newly proposed concept (here, GWIGF) with any known information measure.} In an earlier work by \cite{kharazmi2021jensen}, a new representation for IGF in terms of $H_{k}(X)$ has been provided. Along with the similar arguments, in the next proposition, we establish that the GWIGF in (\ref{eq2.1}) can be represented based on the weighted Shannon entropy of order $k$, defined as
\begin{eqnarray}\label{eq2.8}
H_{k}^{\omega}(X)=\int_{0}^{\infty}\omega(x)(-\log f(x))^{k} f(x)dx.
\end{eqnarray}

\begin{proposition}
	For an RV $X$, we have
	\begin{eqnarray}
	I_{\beta}^{\omega}(X)=\sum_{k=0}^{\infty}\frac{(1-\beta)^{k}}{k!}H_{k}^{\omega}(X),
	\end{eqnarray}
	where $H_{k}^{\omega}(X)(<\infty)$ is given in (\ref{eq2.8}).
\end{proposition}

\begin{proof}
	We obtain
	\begin{eqnarray*}
	I_{\beta}^{\omega}(X)=E[\omega(X)f^{\beta-1}(X)]&=&E\left[\omega(X)e^{-(1-\beta)\log f(X)}\right]\\
	&=&E\left[\omega(X)\sum_{k=0}^{\infty}\frac{(1-\beta)^{k}}{k!}(-\log f(X))^{k}\right]\\
	&=&\sum_{k=0}^{\infty}\frac{(1-\beta)^{k}}{k!}\int_{0}^{\infty}\omega(x)(-\log f(x))^{k}f(x)dx.
	\end{eqnarray*}
Thus, the result is established. 
\end{proof}

Below, we obtain an inequality for the GWIGF given in (\ref{eq2.1}). {Specifically, in the following theorem, we establish lower and upper bounds of the GWIGF in terms of other weighted information generating functions.}
\begin{theorem}\label{th2.1}
	For any real number $\beta\ge 1$, the following inequality holds:
	\begin{eqnarray*}
		\left(I_{\frac{\beta+1}{2}}^{\sqrt{\omega}}(X)\right)^{2}\le 	I_{\beta}^{\omega}(X)\le \sqrt{	I_{2\beta-1}^{\omega^2}(X)}.
	\end{eqnarray*}
\end{theorem}

\begin{proof}
	To establish $\left(I_{\frac{\beta+1}{2}}^{\sqrt{\omega}}(X)\right)^{2}\le 	I_{\beta}^{\omega}(X),$ the Cauchy-Schwartz (CS) inequality is used. For two real integrable functions $h_1(x)$ and $h_2(x)$, the CS inequality is 
	\begin{eqnarray}\label{eq2.10}
	\left(\int_{0}^{\infty}h_{1}(x) h_{2}(x)dx\right)^{2}\le \int_{0}^{\infty}h_{1}^{2}(x)dx  \int_{0}^{\infty}h_{2}^{2}(x)dx,
	\end{eqnarray}
	where all the integrals must exist. Now, set $h_1(x)=\sqrt{\omega(x)}f^{\frac{\beta}{2}}(x)$ and $h_2(x)=f^{\frac{1}{2}}(x)$. Using $h_1(x)$ and $h_2(x)$ in (\ref{eq2.10}), we obtain 
	\begin{eqnarray}\label{eq2.11}
	\left(I_{\frac{\beta+1}{2}}^{\sqrt{\omega}}(X)\right)^{2}=\left(\int_{0}^{\infty}\sqrt{\omega(x)}f^{\frac{\beta+1}{2}}(x)dx\right)^2\le \int_{0}^{\infty} \omega(x)f^{\beta}(x)dx \int_{0}^{\infty}f(x)dx=I_{\beta}^{\omega}(X).
	\end{eqnarray}
	Further, in order to establish $I_{\beta}^{\omega}(X)\le \sqrt{	I_{2\beta-1}^{\omega^2}(X)}$, we use Jensen's inequality for integration. Consider a positive-valued function $g(x)$ such that $\int_{0}^{\infty}g(x)dx=1.$ Then, for a convex function $\psi(\cdot),$ the Jensen's inequality is given by (see \cite{kharazmi2021jensen})
	\begin{eqnarray}\label{eq2.12}
	\psi\left(\int_{0}^{\infty}h(x)g(x)dx\right) \le \int_{0}^{\infty} \psi(h(x))g(x)dx,
	\end{eqnarray}
	where $h(\cdot)$ is a real-valued function. Set $g(x)=f(x)$, $\psi(x)=x^2$ and $h(x)=\omega(x)f^{\beta-1}(x)$. Then, from (\ref{eq2.12}), we obtain 
	\begin{eqnarray}\label{eq2.13}
	\left( I_{\beta}^{\omega}(X)\right)^2=\left(\int_{0}^{\infty}\omega(x)f^{\beta}(x)\right)^2&\le& \int_{0}^{\infty}\omega^2(x)f^{2\beta-2}(x)f(x)dx\nonumber\\
	&=&\int_{0}^{\infty}\omega^2(x)f^{2\beta-1}(x)dx\nonumber\\
	&=& I_{2\beta-1}^{\omega^2}(X).
	\end{eqnarray}
	Thus, the required result follows after combining (\ref{eq2.11}) and (\ref{eq2.13}).  This completes the proof. 
\end{proof}

{We observe that the result in Theorem \ref{th2.1} reduces to Theorem $2$ of \cite{kharazmi2021jensen} when the weight is equal to $1.$} The following example illustrates Theorem \ref{th2.1}. 
\begin{example}\label{ex2.3}
	Suppose $X$ has PDF $f(x)=\lambda e^{-\lambda x},$ $x>0$, $\lambda>0.$ For $\omega(x)=x$, some calculations lead to
	\begin{eqnarray*} I_{\frac{\beta+1}{2}}^{\sqrt{x}}(X)&=&\sqrt{\lambda^{\frac{\beta-2}{2}}\left(\frac{2}{\beta+1}\right)^{\frac{3}{2}}\Gamma\left(\frac{3}{2}\right)},~~
	 I^x_\beta(X)=\frac{\lambda^{\beta-2}}{\beta^2},\nonumber\\ I^{x^2}_{2\beta-1}(X)&=&\frac{\lambda^{2\beta-4}}{(2\beta-1)^3}\Gamma(3).
	 \end{eqnarray*}
\end{example}

In Figure $2(a)$, we plot the graphs of  	$\left(I_{\frac{\beta+1}{2}}^{\sqrt{x}}(X)\right)^{2},$ $I_{\beta}^{x}(X),$ and $\sqrt{	I_{2\beta-1}^{x^2}(X)}$ as in Example \ref{ex2.3}, for $\beta\ge1,$ which clearly justifies the result in Theorem \ref{th2.1}. 
 \begin{figure}[h!]\label{fig3}
	\centering
	\subfigure[]{\label{c1}\includegraphics[height=1.92in]{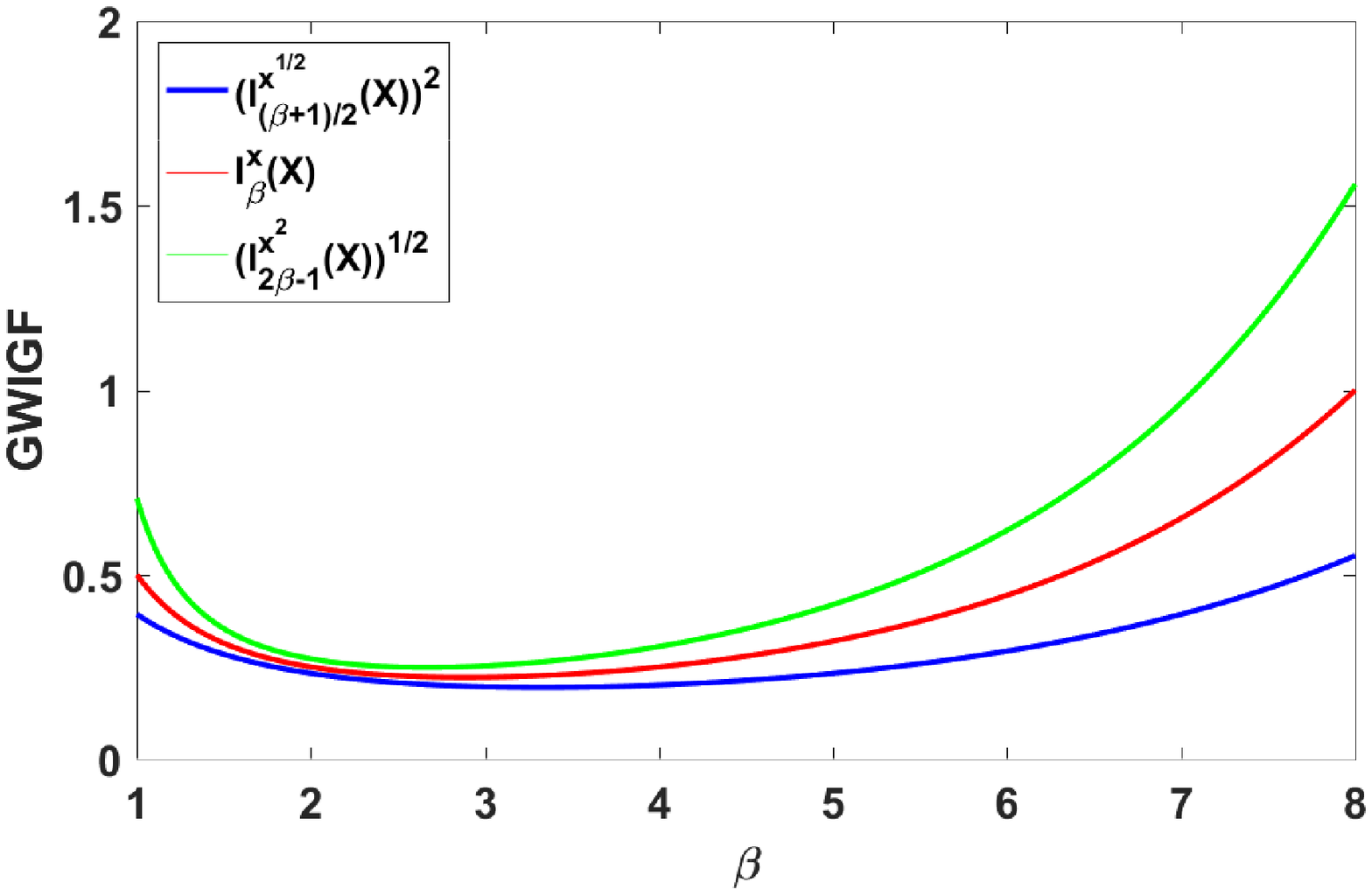}}
	\subfigure[]{\label{c1}\includegraphics[height=1.92in]{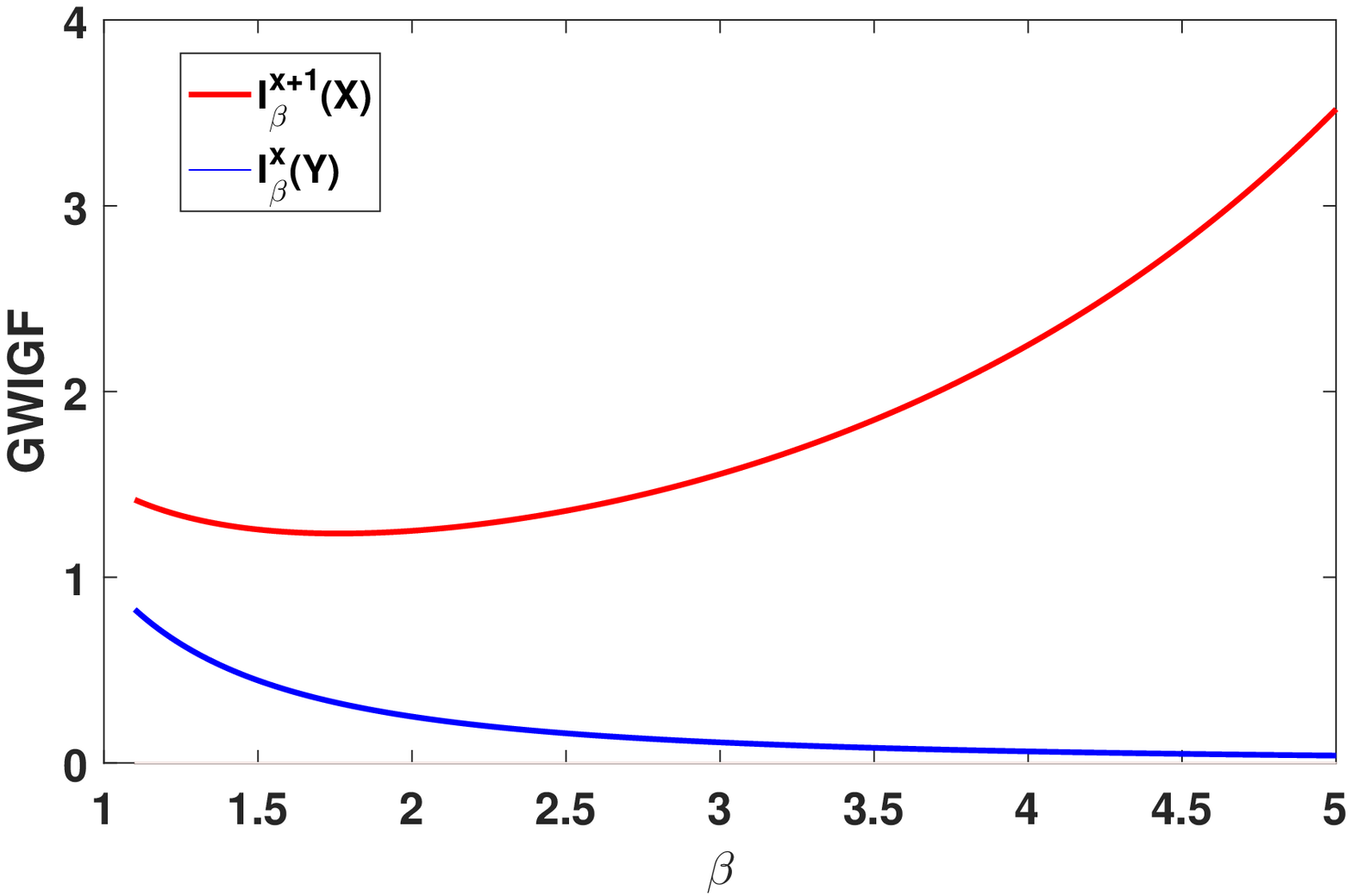}}
	\caption{$(a)$ Graphs of $\left(I_{\frac{\beta+1}{2}}^{\sqrt{x}}(X)\right)^{2},$ $I_{\beta}^{x}(X),$ and $\sqrt{	I_{2\beta-1}^{x^2}(X)}$ as in Example \ref{ex2.3}, for $\beta\ge1,$ when $\lambda=2.$ $(b)$ Graphs of the GWIGFs $I^{x+1}_\beta(X)$ and $I^{x}_\beta(Y)$ for $\lambda_1=2$, $\lambda_2=1$ and $\beta\geq1$  as in Example \ref{ex2.6}}
\end{figure}
Now we provide bounds of GWIGF based on weighted Shannon entropy and the hazard rate of $X$, given by $h(\cdot)=\frac{f(\cdot)}{\bar{F}(\cdot)},$ where $\bar{F}(x)=P(X>x).$ {We recall that the hazard/failure rate is defined as the (instantaneous) rate of failure for the survivors to time $t$ during the next instant of time. It is an important concept in survival analysis.}

\begin{theorem}\label{th2.2}
	Suppose $X$ has PDF $f(\cdot)$. The survival function of $X$ is denoted by $\bar{F}(\cdot)$. Then, for any $\beta\ge1,$ we have 
	\begin{eqnarray}
	L\le I_{\beta+1}^{\omega}(X)\le U,
	\end{eqnarray}
	where $L=\max\{0,E(\omega(X))-\beta H^{\omega}(X)\}$ and $U=E[\omega(X)(h(X))^{\beta}]$.
\end{theorem}

\begin{proof}
	The proof follows using similar arguments as in Theorem $4$ of \cite{kharazmi2021jensen}. 
\end{proof}

{ In the following example, we obtain the bounds of the GWIGF, stated in Theorem \ref{th2.2}.}
\begin{example}
	Let $X$ follow Weibull distribution with PDF 
	\begin{align}
	f(x)=\begin{cases}
	\frac{\delta}{\theta}(\frac{x}{\theta})^{\delta-1}e^{-(\frac{x}{\theta})^{\delta}},~~\text{if} ~x>0,\\
	0,~~\text{otherwise},
	\end{cases} 
	\end{align}	  
	where $\delta>0$ and $\theta>0.$ Considering $\omega(x)=x$, for this distribution, we obtain 
		\begin{eqnarray*}
			E[X]-\beta H^{x}(X)&=&\theta\Gamma(2)\left\{1-\beta\log\left(\frac{\delta}{\theta}\right)\right\}+\theta\beta\Gamma\left(2+\frac{1}{\delta}\right)-\theta\beta(\delta-1)\psi^*\left(\frac{1}{\delta}+1\right),\nonumber\\
			 I_{\beta}^{x}(X)&=&\frac{\delta^{\beta-1}\theta^{2-\beta}}{\beta^(\frac{\beta\delta-\beta+2}{\delta})}\Gamma\left(\frac{\beta\delta-\beta+1}{\delta}\right),\nonumber\\
		E[Xh^{\beta}(X)]&=&\frac{\delta^{\beta}}{\theta^{\beta-1}}\Gamma\left(\frac{\beta+1}{\delta}+1\right),
	\end{eqnarray*}
where $\beta\delta-\beta+1>0$ and $\psi^{*}(\cdot)$ is known as the digamma function. 
\end{example}

The result in the following remark immediately follows  from Theorem \ref{th2.2}, which gives bounds of the weighted extropy. 
\begin{remark}
	For $\beta=1$, we have 
	$$-\frac{1}{2}E(\omega(X)h(X))\le J^{\omega}(X)\le -\frac{1}{2}\max\{0,E(\omega(X)-H^{\omega}(X))\}.$$
\end{remark}

{Sometimes, it is very hard to obtain explicit expression of the GWIGF of a random
variable of interest. To overcome this difficulty, certain transformation can be employed for computing GWIGF
of a transformed random variable (new distribution) in terms of the GWIGF of a known
distribution. The following result provides an useful insight in this regard.} Next, we study the effect of GWIGF under an increasing/decreasing, differentiable and invertible function. 
\begin{theorem}\label{th2.3}
	Suppose an RV $X$ has PDF $f(\cdot)$. Then, for any real-valued monotone, differentiable and invertible function $\zeta(x)$, we have 
	 \begin{equation*}
I_{\beta}^{\omega}(\zeta(X))=\left\{
	\begin{array}{ll}
	\displaystyle\int_{\zeta^{-1}(0)}^{\zeta^{-1}(\infty)}\frac{\omega\big(\zeta(x)\big)f^\beta(x)}{\big(\zeta^{'}(x)\big)^{\beta-1}}dx,~\text{if  $\zeta$  is  strictly  increasing;}
	\\
	\\
	(-1)^{\beta-1}\displaystyle\int_{\zeta^{-1}(0)}^{\zeta^{-1}(\infty)}\frac{\omega\big(\zeta(x)\big)f^\beta(x)}{\big(\zeta^{'}(x)\big)^{\beta-1}}dx,~\text{if  $\zeta$  is  strictly  decreasing.}
	\end{array}
	\right.
	\end{equation*}
\end{theorem}

\begin{proof} Here, we prove the result when $\zeta(\cdot)$ is increasing. The proof is similar when $\zeta(\cdot)$ is decreasing. Let $\zeta$ be strictly increasing. Then,
	after some simple calculations, the PDF of $\zeta(X)$ is
	\begin{eqnarray}\label{eq2.16}
	f_{\zeta}(u)=\frac{f(\zeta^{-1}(u))}{\zeta'(\zeta^{-1}(u))},~~0<u<\infty.
	\end{eqnarray}
	Now, using (\ref{eq2.16}) in (\ref{eq2.1}), the required result readily follows. 
\end{proof}

The following example illustrates Theorem \ref{th2.3}.
\begin{example}
	Suppose $X$ follows Pareto-I distribution with CDF $F(x)=1-x^{-c},~x>1$ and $c>0.$ Consider $\zeta(X)=X-1,$ which follows Lomax distribution. Using Theorem \ref{th2.3}, we obtain for $c=2$ as
	$$I_{\beta}^{x}(\zeta(X)=X-1)=\frac{2^{\beta}}{(1-3\beta)(2-3\beta)},$$
	which can be verified using direct calculations. 
\end{example}

{ The concept of stochastic orders play an important role in many fields of research. For example, in risk theory, the usual stochastic order is useful for a decision maker to choose better risk. Further, in actuarial science, the decision maker uses convex order to select preferable risk if the maker is risk averse. For details, please refer to \cite{shaked2007stochastic}.} We study stochastic comparison result between GWIGFs of two RVs by considering different weights. In this direction, the notion of dispersive order, denoted by $\le_{disp}$ (see Equation $(3.B.1)$ of \cite{shaked2007stochastic}) is required. {We recall that the dispersive order is used to compare the variability of two risk random variables.}

\begin{theorem}\label{th2.4}
	Suppose $X$ and $Y$ have CDFs $F(\cdot)$ and $G(\cdot)$ and PDFs $f(\cdot)$ and $g(\cdot)$, respectively.  Let $\omega_1(\cdot)$ be increasing. Then, for $\beta\ge1$,
$\omega_1(x)\geq (\le) \omega_2(x)$ and $X\leq_{disp}(\geq_{disp})Y$ together imply $I^{\omega_1}_\beta(X)\geq (\le)I^{\omega_2}_\beta(Y)$.
\end{theorem} 

\begin{proof}
	We know that $X\leq_{disp}(\geq_{disp})Y$ implies $F^{-1}(u)\geq (\le)G^{-1}(u),$ for $u\in(0,1)$.
	Further, under the assumptions made, $\omega_1(x)$ is increasing and $\omega_1(x)\ge(\le)\omega_2(x)$. Thus, we obtain $\omega_1(F^{-1}(u))\geq(\le) \omega_1(G^{-1}(u))\geq(\le)\omega_2(G^{-1}(u))$. 
	Now, for $\beta\ge1,$
	\begin{eqnarray*}
		I^{\omega_1}_\beta(X)
		=\int^1_0\omega_1(F^{-1}(u))f^{\beta-1}(F^{-1}(u))du
		\ge(\le)\int^1_0\omega_2(G^{-1}(u))g^{\beta-1}(G^{-1}(u))du
		=I^{\omega_2}_\beta(Y).
	\end{eqnarray*}
	Hence, the theorem is established.
\end{proof}  

\begin{remark}
	The result in Theorem \ref{th2.4} also holds if two weights are same, that is, $\omega_{1}(x)=\omega_{2}(x).$ { Further, note that if $\omega_{1}(x)=\omega_{2}(x)=1$, Theorem \ref{th2.4} becomes Theorem $6(ii)$ of \cite{kharazmi2021jensen}.}
\end{remark}

\begin{remark}
	It is well-known that for  two RVs $X$ and $Y$, $X\le_{hr}Y$ (see Equation $(1.B.2)$ of \cite{shaked2007stochastic})  implies $X\le_{disp}Y$ if either $X$ or $Y$ has decreasing failure rate (DFR). Thus, the result in Theorem \ref{th2.4} also holds if we replace the condition dispersive ordering by the hazard rate ordering with DFR property of either $X$ or $Y$. 
\end{remark}

The following example is considered for illustrating Theorem \ref{th2.4}.

\begin{example}\label{ex2.6}
	Suppose $X$ and $Y$ have CDFs $F(x)=1-e^{-\lambda_1 x}$ and $G(x)=1-e^{-\lambda_2 x}$; PDFs $f(x)=\lambda_1 e^{-\lambda_1 x}$ and $g(x)=\lambda_2 e^{-\lambda_2 x},$ respectively, where $x>0,~\lambda_1,\lambda_2>0$. Suppose $\lambda_1\ge \lambda_2.$ Here, $f(F^{-1}(u))=\lambda_1(1-u)$ and $g(G^{-1}(u))=\lambda_2(1-u)$, where $0<u<1.$ Clearly, $f(F^{-1}(u))\ge g(G^{-1}(u))$ implies $X\le_{disp}Y$ holds. Further, assume that $\omega_1(x)=x+1$ and $\omega_{2}(x)=x,$ for $x>0.$  The GWIGFs are obtained as
	\begin{eqnarray}\label{eq2.17}
		I^{x+1}_\beta(X)=\frac{\lambda_1^\beta}{(\lambda_1 \beta)^2}(\lambda_1 \beta+1) ~~and ~~ I^{x}_\beta(Y)=\frac{\lambda_2^\beta}{(\lambda_2 \beta)^2},
	\end{eqnarray}
  which are plotted in Figure $2(b)$,  showing that  $I^{x+1}_\beta(X)\geq I^{x}_\beta(Y)$, for $\beta\geq1$.
\end{example}

The concepts of varentropy and weighted varentropy are of recent interest to study the variation of the information content and { weighted information content of a RV, respectively. It well known that the weighted varentropy is not affine invariant. However, the varentropy is affine invariant. Moreover, the varentropy does not
distinguish variability in the information content of a random process in different time points. This drawback can be resolved if we use weighted varentropy. Note that the use of weighted varentropy is also motivated by the necessity, arising in various communication and transmission problems, of expressing the usefulness of events with a variance measure.} Here, we establish a connection between the weighted varentropy and varentropy with proposed GWIGF. Define,
   \begin{eqnarray}
\ce{_{k}D^{\omega^2}_\beta}(X)=\frac{\partial^k}{\partial \beta^k}I^{\omega^2}_\beta(X)~~\mbox{and} ~~\ce{_{k-1}D^{\omega}_\beta}(X)=\frac{\partial^{k-1}}{\partial \beta^{k-1}}I^{\omega}_\beta
(X).
\end{eqnarray}
Using these notations, the weighted varentropy of $X$ with weight $\omega(x)$ is 
\begin{eqnarray}\label{eq2.19}
VE^\omega(X)&=&\ce{_{2}D^{\omega^2}_1}(X)-\big(\ce{_{1}D^{\omega}_1}(X)\big)^2\nonumber\\
&=&\int_{0}^{\infty}\omega^{2}(u)f(u)(\log f(u))^{2}du-\left(\int_{0}^{\infty}\omega(u)f(u)\log f(u)du\right)^2\nonumber\\
&=&E\left((\omega(X)\log f(X))^{2}\right)-\left(E(\omega(X)\log f(X))\right)^2.
\end{eqnarray}
The weighted varentropy due to \cite{saha2024weighted} can be obtained from (\ref{eq2.19}) for $\omega(x)=x.$ Further, if we assume $\omega(x)=1$, the varentropy can be deduced from (\ref{eq2.19}). It is worth mentioning that one of the early appearances of the varentropy is when it was characterised as the minimal coding variance in \cite{kontoyiannis1997second}, and then more recently when it was identified as the dispersion in source coding in \cite{kontoyiannis2013optimal}. For some notions on varentropy, see \cite{fradelizi2016optimal}.
\\

We end this section with an upper bound for GWIGF with $\omega(x)=x$ of the sum of two independent RVs. { One can find applications of the sum of two random variables in different fields. For example, in risk theory the sum of two random variables represents the combined risk of two assets. In signal processing, the sum of the random variables can represent the combined effect of noise in a communication channel. In reliability analysis, the sum of two random variables represent the total lifetime of a system having several components. To obtain upper bound, here we use the concept of convolution, which has many applications in information theory, physics and engineering. We recall that de Bruijn identity relates entropy with Fisher information, which can be deduced as a special case of an immediate generalization of Price’s theorem, which is a tool used in the analysis of
nonlinear memoryless systems having Gaussian inputs. For details,  see \cite{park2012equivalence}.} The PDF of the sum $Z$ of two independent RVs $X$ and $Y$ with respective PDFs $f(\cdot)$ and $g(\cdot)$ is 
\begin{eqnarray}\label{eq2.18}
f_{Z}(z)=\int_{0}^{z}f(x)g(z-x)dx,~~z>0.
\end{eqnarray}

\begin{theorem}\label{th2.5}
	Suppose two independent RVs $X$ and $Y$ have PDFs $f(\cdot)$ and $g(\cdot)$, respectively. Then, for $Z=X+Y$ and $\omega(x)=x,$ we have
	\begin{eqnarray}
	I_{\beta}^{x}(Z)\le I_{\beta}(X)I_{\beta}^{y}(Y)+I_{\beta}^{x}(X)I_{\beta}(Y),~~\beta\ge1.
	\end{eqnarray}
\end{theorem}

\begin{proof}
From (\ref{eq2.1}) and (\ref{eq2.18}), we have 
\begin{eqnarray}\label{eq2.20}
I_{\beta}^x(Z)&=&\int_{0}^{\infty}z\left(\int_{0}^{z}f(x)g(z-x)dx\right)^{\beta}dz\nonumber\\
&\le&\int_{0}^{\infty}z \left(\int_{0}^{z}f^\beta(x)g^\beta(z-x)dx\right)dz\nonumber\\
&=& \int_{0}^{\infty}f^{\beta}(x) \left(\int_{x}^{\infty}z g^{\beta}(z-x)dz\right)dx\nonumber\\
&=& \int_{0}^{\infty}f^{\beta}(x)\left(\int_{0}^{\infty} (z+x)g^{\beta}(z)dz\right)dx,
\end{eqnarray}	
where the inequality follows from Jensen's inequality. Now, the desired result readily follows from (\ref{eq2.20}). 
\end{proof}

\begin{remark}
	For independently and identically distributed RVs $X$ and $Y$, we have from Theorem \ref{th2.5} as $$I_{\beta}^{x}(X+Y)\le 2I_{\beta}(X)I_{\beta}^{x}(X).$$
\end{remark}

\section{General weighted relative information generating function}
The RIGF was introduced by \cite{guiasu1985relative}. The authors showed that the KL divergence can be generated from the RIGF after evaluating the derivative at $1$. This section is devoted to introduce general weighted relative information generating function (GWRIGF) and study its properties. {Recall that employing weight function in statistical analysis emphasises importance to certain observations, which are more valuable than others.}

\begin{definition}\label{def3.1}
	Suppose $X$ and $Y$ have PDFs $f(\cdot)$ and $g(\cdot),$ respectively. Then, for any  non-negative real-valued measurable function $\omega(x)$ and $\beta\ge1$, the GWRIGF is 
	\begin{eqnarray}\label{eq3.1}
	R^\omega_\beta(X,Y)=\int_{0}^{\infty}\omega(x)f^\beta(x)g^{1-\beta}(x)dx,
	\end{eqnarray}
	provided the integral exists.
\end{definition}
Clearly, $R_{\beta}^{\omega}(X,Y)\ge0.$ {For $\omega(x)=1,$ the GWRIGF in (\ref{eq3.1}) becomes RIGF proposed by \cite{guiasu1985relative}.} Further, when $g(x)=1$, that is $Y$ has uniform distribution,  (\ref{eq3.1}) reduces to GWIGF given in (\ref{eq2.1}). Differentiation of (\ref{eq3.1}) $k$-times with respect to $\beta$ is
\begin{eqnarray}\label{eq3.2}
\frac{\partial^k R^\omega_\beta(X,Y)}{\partial \beta^k} =\int_{0}^{\infty}\omega(x)\bigg(\log \frac{f(x)}{g(x)}\bigg)^k\bigg(\frac{f(x)}{g(x)}\bigg)^\beta g(x)dx,
\end{eqnarray}
provided that the integral converges. The following observations can be obtained in a straightforward manner:
\begin{itemize}
	\item $R^\omega_\beta(X,Y)|_{\beta=1}=E_{f}(\omega(X)),$ $R^\omega_\beta(X,Y)|_{\omega(x)=1,\beta=1}=1;$ 
	\item $\frac{\partial R^\omega_\beta(X,Y)}{\partial \beta}|_{\beta=1}=KL^{\omega}(X,Y),$ $\frac{\partial R^\omega_\beta(X,Y)}{\partial \beta}|_{\omega(x)=1,\beta=1}=KL(X,Y);$
	\item $R^\omega_\beta(X,Y)=R_{1-\beta}^{\omega}(Y,X)$,
\end{itemize}
where $KL(X,Y)=\int_{0}^{\infty}f(u)\log \frac{f(u)}{g(u)}du$ is the well-known KL divergence (see \cite{kullback1951information}) and  $KL^{\omega}(X,Y)=\int_{0}^{\infty}\omega(u)f(u)\log \frac{f(u)}{g(u)}du$ is known as the weighted KL divergence (see \cite{suhov2016basic}). The following example deals with the closed-form expressions of the GWRIGF between two exponential and two Pareto distributions.

\begin{example}\label{ex3.1}~
	\begin{itemize}
	\item[(i)] Let $X$ and $Y$ follow exponential distributions with respective rate parameters $\lambda_1$ and $\lambda_2.$  Further, let $\omega(x)=x^{m},~x>0,~m>0$. After some calculations, we get
	\begin{eqnarray}\label{eq3.3}
	R_{\beta}^{x^m}(X,Y)=\frac{\Gamma(m+1)\lambda_{1}^{\beta}\lambda_{2}^{1-\beta}}{\{\lambda_1\beta+\lambda_{2}(1-\beta)\}^{m+1}},~\beta\ge1.
	\end{eqnarray}
	Further, 
		\begin{eqnarray}\label{eq3.4}
	R_{\beta}^{x^m}(Y,X)=\frac{\Gamma(m+1)\lambda_{2}^{\beta}\lambda_{1}^{1-\beta}}{\{\lambda_2\beta+\lambda_{1}(1-\beta)\}^{m+1}},~\beta\ge1.
	\end{eqnarray}
	From (\ref{eq3.3}) and (\ref{eq3.4}), clearly $R_{\beta}^{x^m}(X,Y)\ne R_{\beta}^{x^m}(Y,X)$ in general, that is the GWRIGF is not symmetric. 
	
	\item[(ii)] Let $X$ and $Y$ follow Pareto distributions with PDFs $f(x)=\frac{\alpha_1}{x^{\alpha_1+1}},~x>1$ and $g(x)=\frac{\alpha_2}{x^{\alpha_2+1}},~x>1$, respectively. Then, for $\omega(x)=x^{m},~x>0,~m>0,$ we compute GWRIGF as
	\begin{eqnarray}
	R_{\beta}^{x^{m}}(X,Y)=\frac{\alpha_1^\beta \alpha_2^{1-\beta}}{(\alpha_1+1)\beta+(\alpha_2+1)(1-\beta)-1-m},~\beta\ge1,
	\end{eqnarray}
	provided $m-(\alpha_1+1)\beta-(\alpha_2+1)(1-\beta)+1<0.$
	\end{itemize}
\end{example}

To see the behaviour of GWRIGF with respect to $\beta$, we plot (\ref{eq3.3}) in Figure $3$, for $m=1.$  
\begin{figure}[h]
	\begin{center}
		\includegraphics[height=3.5in]{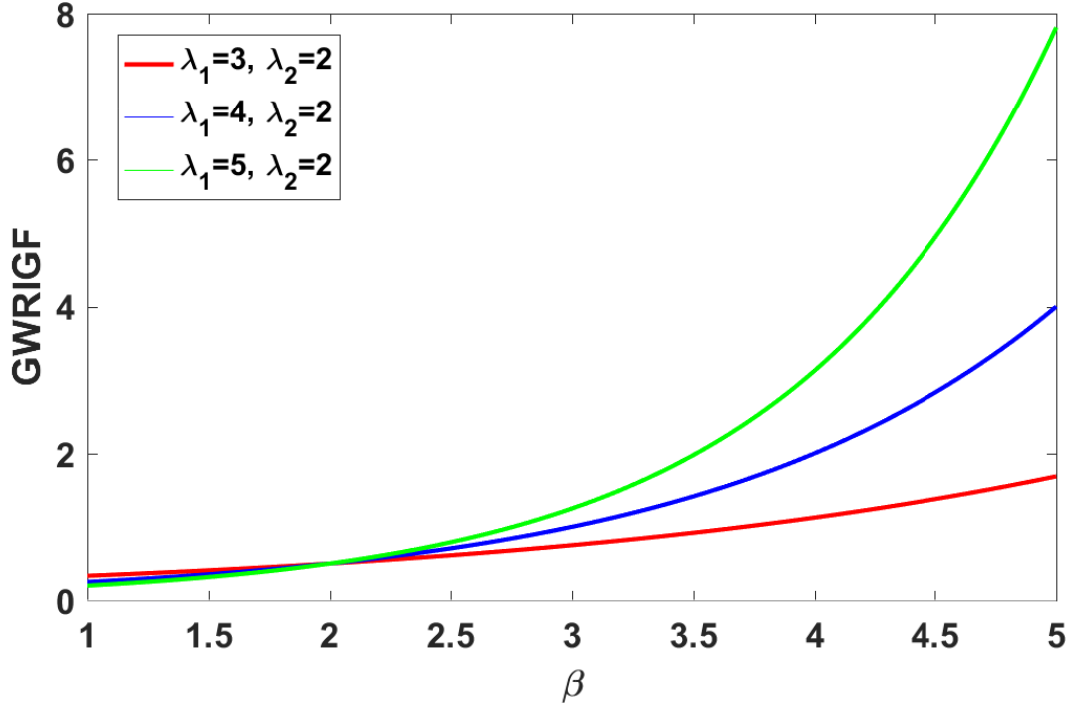}
		\caption{Plots of the GWRIGF of exponential distributions for different values of $\lambda_1$ and $\lambda_2$ in Example \ref{ex3.1}.}	
	\end{center}
\end{figure}

{ Sometimes, it is required to study the information measures for two transformed random variables, of which the density functions are difficult to derive. The following result is useful to obtain the effect of GWRIGF for two transformed random variables.} Here, we consider strictly monotone transformations, {which have wide applications in several fields, such as signal processing, machine learning and finance.}

\begin{theorem}\label{th3.1}
	Suppose $X$ and $Y$ have PDFs $f(\cdot)$ and $g(\cdot),$ respectively and $\psi(\cdot)$ is strictly monotonic, differential and invertible function. For weight $\omega(x)\ge0,$
	\begin{equation*}
	R^\omega_\beta(\psi(X),\psi(Y))=\left\{
	\begin{array}{ll}
	\int_{\psi^{-1}(0)}^{\psi^{-1}(\infty)}\omega(\psi(x))f^\beta(x)g^{1-\beta}(x)dx,~~\text{if  $\psi$  is  strictly  increasing;}
	\\
	\\
	-\int_{\psi^{-1}(0)}^{\psi^{-1}(\infty)}\omega(\psi(x))f^\beta(x)g^{1-\beta}(x)dx,~~\text{if  $\psi$  is  strictly  decreasing.}
	\end{array}
	\right.
	\end{equation*}
\end{theorem}

\begin{proof}
	The PDFs of $\psi(X)$ and $\psi(Y)$ are respectively obtained as
	\begin{eqnarray}\label{eq3.5}
	f_{\psi}(x)=\Big|\frac{1}{\psi^{'}(\psi^{-1}(x))}\Big| f(\psi^{-1}(x))~\mbox{and}~g_{\psi}(x)=\Big|\frac{1}{\psi^{'}(\psi^{-1}(x))}\Big| g(\psi^{-1}(x)),~x\in(\psi(0),\psi(\infty)).
	\end{eqnarray}
Firstly, we consider that $\psi(x)$ is strictly increasing. From (\ref{eq3.1}), using (\ref{eq3.5}) we obtain
	\begin{eqnarray*}
		R^\omega_\beta(\psi(X),\psi(Y))&=&\int_{0}^{\infty}\omega(x)f^\beta_\psi(x)g^{1-\beta}_\psi(x)dx\\
		&=&\int_{0}^{\infty}\omega(x)\left(\frac{f(\psi^{-1}(x))}{\psi^{'}(\psi^{-1}(x))}\right)^\beta \left(\frac{g(\psi^{-1}(x))}{\psi^{'}(\psi^{-1}(x))}\right)^{1-\beta}dx\\
		&=&\int_{\psi^{-1}(0)}^{\psi^{-1}(\infty)}\omega(\psi(x))f^\beta(x) g^{1-\beta}(x)dx.
	\end{eqnarray*}
	The proof for other case is similar, and hence not presented here. 
\end{proof} 

The following remark is straightforward from Theorem \ref{th3.1}. {Here, we consider scale and location transformed random variables to see the effect of the GWRIGF.}

 \begin{remark}~~
For  $a,b>0$, we obtain
		$$R^\omega_\beta(aX,aY)=\int_{0}^{\infty}\omega(ax)f^\beta(x)g^{1-\beta}(x)dx~\mbox{and}~
		R^\omega_\beta(X+b,Y+b)=\int_{-b}^{\infty}\omega(x+b)f^\beta(x)g^{1-\beta}(x)dx.$$
\end{remark}

To illustrate Theorem \ref{th3.1}, the following example is considered.

\begin{example}Consider $X$ and $Y$ as in Example \ref{ex3.1}$(i)$.
	\begin{itemize}
	\item[(i)] Let $\psi(x)=\sqrt{x},~x>0$ and weight $\omega(x)=x.$ Then, $\sqrt{X}$ and $\sqrt{Y}$ follow Weibull distributions with common shape parameter $2$ and different scale parameters $\frac{1}{\lambda_1}$ and $\frac{1}{\lambda_2},$ respectively. Using Theorem \ref{th3.1}, we can easily compute the GWRIGF of two Weibull RVs, which is given by 
	\begin{eqnarray*}
		R^\omega_\beta\left(\sqrt{X},\sqrt{Y}\right)=\frac{\sqrt{\pi}}{2}\frac{\lambda_1^\beta \lambda_2^{1-\beta}}{\{\beta\lambda_1+(1-\beta)\lambda_2\}^{\frac{3}{2}}},~~\beta\lambda_1+(1-\beta)\lambda_2>0.
	\end{eqnarray*} 

\item[(ii)]  Consider the transformation $\psi(x)=\frac{1}{x},~x>0$ and weight $\omega(x)=\frac{1}{x}.$ It can be shown that $\frac{1}{X}$ and $\frac{1}{Y}$ follow inverted exponential distributions with respective PDFs $f(x)=\frac{\lambda_1}{x^2}e^{-\frac{\lambda_1}{x}}$ and $g(x)=\frac{\lambda_2}{x^2}e^{-\frac{\lambda_2}{x}}$, $x>0,~\lambda_1,\lambda_2>0.$  Using Theorem \ref{th3.1}, we can derive the GWRIGF of two inverted exponential RVs as
\begin{eqnarray*}
	R^\omega_\beta\left(\frac{1}{X},\frac{1}{Y}\right)=\frac{\lambda_1^{\beta}\lambda_2^{1-\beta}}{(\beta\lambda_1+(1-\beta)\lambda_2)^2},~~\beta\lambda_1+(1-\beta)\lambda_2>0.
\end{eqnarray*} 
\end{itemize}
\end{example} 

We end this section with an interesting observation that weighted KL $J$-divergence between $X$ and $Y$, denoted by $J^{\omega}(X,Y)$  can be expressed  in terms of GWRIGF as follows:
\begin{eqnarray}
J^{\omega}(X,Y)&=&\int_{0}^{\infty}\omega(u)\left(f(u)\log\frac{f(u)}{g(u)}+g(u)\log\frac{g(u)}{f(u)}\right)du\nonumber\\
&=&\frac{\partial }{\partial \beta} R_{\beta}^{\omega}(X,Y)|_{\beta=1}+\frac{\partial }{\partial \beta} R_{\beta}^{\omega}(Y,X)|_{\beta=1}.
\end{eqnarray}

\section{GWIGF and GWRIGF for escort, generalized escort and $(r,\gamma)$-mixture models}
For the characterization of chaos and multifractals, the escort distribution was introduced in statistical physics. In the context of multifractals, the escort distributions have been proposed as an operational tools. Please refer to \cite{bercher2008some} and \cite{bercher2011escort} for  usefulness of the escort distribution in statistical thermodynamics. {\cite{abe2003geometry} provided quantitative evidence for the inappropriateness of utilising the original distribution in stead of the escort distribution for determining the expectation values of physical quantities in nonextensive statistical mechanics (see \cite{tsallis1988possible}, \cite{abe2000remark}, \cite{abe2001nonextensive}) based on the Tsallis entropy using a relative divergence (Kullback-Leibler).} \cite{pennini2004escort} discussed about some geometrical properties of the escort distribution. Due to importance of the escort distribution, it has been widely used not only in statistical physics but also in different fields of research. For example, in coding theory, the escort distributions are useful in computing optimal codewords (see \cite{pennini2004escort}). In digital image analysis, the escort distributions are also useful (see \cite{ampilova2021using}). In this section, we study GWIGF and GWRIGF for escort distributions. Suppose an RV $X$ has PDF $f(\cdot).$ The PDF of the escort distribution with baseline density $f(\cdot)$, is given by
\begin{eqnarray}\label{eq4.1}
f_{e,\alpha}(u)=\frac{f^{\alpha}(u)}{\int_{0}^{\infty}f^{\alpha}(u)du},~u>0,
\end{eqnarray}
where $\alpha$ is any positive real number, know as the order of the distribution. Here, $\alpha$ is useful to explore different types of regions of $f_{e,\alpha}(\cdot).$ For example, when $\alpha>1,$ singular regions are amplified. Less singular regions are accentuated  for $\alpha<1.$ For details, please see \cite{bercher201p}. Another important family of distribution is generalized escort distribution which is an extension of the escort distribution given in (\ref{eq4.1}). For two RVs $X$ and $Y$ with corresponding density functions  $f(\cdot)$ and $g(\cdot)$, the PDF of  the generalized escort distribution of order $\alpha>0$  is 
\begin{eqnarray}\label{eq4.2}
f_{e,\alpha}^{*}(x)=\frac{f^\alpha(x)g^{1-\alpha}(x)}{\int_{0}^{\infty} f^\alpha(x)g^{1-\alpha}(x)dx},~x>0.
\end{eqnarray}
For details, one may refer to \cite{Beck1993}. \cite{Beck1993} demonstrated that the escort and generalized escort distributions play an important role for determining several types of  similarities between  thermodynamics and chaos theory. In the following consecutive theorems, we present the general weighted generating functions for escort and generalized escort distributions.

\begin{theorem}\label{th4.1}
	Suppose $X$ has PDF $f(\cdot)$. Then, GWIGF  of the escort RV of order $\alpha$, denoted by $X_{e,\alpha}$ is expressed as
	\begin{eqnarray*}
		I^\omega_\beta(X_{e,\alpha}) = \frac{I^\omega_{\alpha \beta}(X)}{\big(I_\alpha(X)\big)^\beta}, ~~\alpha>0,~ \beta\ge1,
	\end{eqnarray*}
where $I_{\alpha}(X)$ is the IGF given in (\ref{eq1.3}).
\end{theorem}

\begin{proof}
	From (\ref{eq2.1}) and (\ref{eq4.1}), we obtain 
	\begin{eqnarray*}
		I^\omega_{\beta }(X_{e,\alpha})=\int_{0}^{\infty}\omega(x)f^\beta_{e,\alpha}(x)dx
		=\frac{\int_{0}^{\infty}\omega(x)f^{\alpha \beta}(x)dx}{\big(\int_{0}^{\infty} f^\alpha(x)dx\big)^\beta}
		=\frac{ I^\omega_{\alpha \beta}(X)}{(I_\alpha(X))^\beta}.
	\end{eqnarray*}
	Hence the result is obtained.
\end{proof} 
{We observe from Theorem \ref{th4.1} that the GWIGF of escort distribution can be easily evaluated  using GWIGFs of a parent distribution.

}

\begin{theorem}
	Suppose  $X$ and $Y$ have respective PDFs $f(\cdot)$ and $g(\cdot)$. Further, let $X_{e,\alpha}^*$ denote the RV corresponding to the PDF in (\ref{eq4.2}). Then, 
	\begin{eqnarray}\label{eq4.3}
	I^\omega_\beta(X_{e,\alpha}^*)=\frac{[I_\beta(X)]^\alpha [I_\beta(Y)]^{1-\alpha}}{[R_\alpha(X,Y)]^\beta}R^\omega_\alpha(X_{e,\beta},Y_{e,\beta}),
	\end{eqnarray}
	where $R_\alpha(X,Y)$ and $R^\omega_\alpha(X_{e,\beta},Y_{e,\beta})$ are the RIGF and GWRIGF, respectively, and $X_{e,\beta}$ and $Y_{e,\beta}$ are escort distributed RVs of order $\beta$ with corresponding baseline densities  $f(\cdot)$ and $g(\cdot)$.
\end{theorem}

\begin{proof}
	From (\ref{eq2.1}) and (\ref{eq4.2}), we get GWIGF of generalized escort distribution as
	\begin{eqnarray*}
	I^\omega_\beta(X_{e,\alpha}^*)&=& \int_{0}^{\infty}\omega(x)(f_{e,\alpha}^*(x))^{\beta}dx\nonumber\\
	&=& \int_{0}^{\infty}\omega(x)\frac{f^{\alpha \beta}(x)g^{\beta(1-\alpha)}(x)}{\big\{\int_{0}^{\infty} f^\alpha(x)g^{1-\alpha}(x)dx\big\}^\beta}dx\nonumber\\
	&=&\frac{[I_\beta(X)]^\alpha [I_\beta(Y)]^{1-\alpha}}{[R_\alpha(X,Y)]^\beta}\int_{0}^{\infty}\omega(x)
	\bigg\{\frac{f^\beta(x)}{\int_{0}^{\infty} f^\beta(x)dx}\bigg\}^\alpha\bigg\{\frac{g^\beta(x)}{\int_{0}^{\infty} g^\beta(x)dx}\bigg\}^{1-\alpha}dx\nonumber\\
	&=&\frac{[I_\beta(X)]^\alpha [I_\beta(Y)]^{1-\alpha}}{[R_\alpha(X,Y)]^\beta}\int_{0}^{\infty}\omega(x)(f_{e,\beta}(x))^\alpha(g_{e,\beta}(x))^{1-\alpha}dx\nonumber\\
	&=&\frac{[I_\beta(X)]^\alpha [I_\beta(Y)]^{1-\alpha}}{[R_\alpha(X,Y)]^\beta}R^\omega_\alpha(X_{e,\beta},Y_{e,\beta}).
	\end{eqnarray*}
	Thus, the result is established.
\end{proof}

The $(r,\gamma)$-mixture distribution of two PDFs $f_1(\cdot)$ and $f_{2}(\cdot)$, denoted by $f_{M(r,\gamma)}(x),$ is given by 
\begin{eqnarray}\label{eq4.4}
f_{M(r,\gamma)}(x)=\frac{[r f_{1}^{\gamma}(x)+(1-r)f_{2}^{\gamma}(x)]^{\frac{1}{\gamma}}}{\int_{0}^{\infty}[r f_{1}^{\gamma}(x)+(1-r)f_{2}^{\gamma}(x)]^{\frac{1}{\gamma}}dx},~x>0,~\gamma>0,~0<r<1.
\end{eqnarray}
Here, $M(r,\gamma)$ denotes the $(r,\gamma)$-mixture RV with PDF given in (\ref{eq4.4}). Further, we assume that $X_{1}$ and $X_{2}$ have PDFs $f_{1}(\cdot)$ and $f_{2}(\cdot),$ respectively. For details about the $(r,\gamma)$-mixture model, we refer to \cite{van2014renyi}, who  introduced this mixture model for $m$ number of probability distributions. \cite{kharazmi2023optimal} discussed three optimization problems, for which the optimal solutions are the $(r,\gamma)$-mixture density function given in (\ref{eq4.4}). These authors also obtained IGF and RIGF of the mixture density function given in (\ref{eq4.4}). In the following consecutive theorems, we study GWIGF and GWRIGF of the $(r,\gamma)$-mixture model. 

\begin{theorem}\label{th4.3}
	The GWIGF of the $(r,\gamma)$-mixture RV $M(r,\gamma)$ is expressed as
	\begin{eqnarray}
		I_{\beta}^{\omega}(M(r,\gamma))=\frac{I_{\frac{\beta}{\gamma}}^{\omega}(X_{\Gamma})}{\left(I_{\frac{1}{\gamma}}(X_{\Gamma})\right)^{\beta}},
	\end{eqnarray}
where $X_{\Gamma}$ is another RV associated with the mixture distribution $f_{\Gamma}(x)=\Gamma f_{(\beta,1)}(x)+(1-\Gamma)f_{(\beta,2)}(x)$ based on the escort PDFs corresponding to $f_{1}(x)$ and $f_{2}(x),$ where 
\begin{eqnarray}
\Gamma=\frac{r I_{\beta}(X_{1})}{r I_{\beta}(X_1)+(1-r)I_{\beta}(X_{2})}~\mbox{and}~f_{(\beta,i)}(x)=\frac{f_{i}^{\beta}(x)}{\int_{0}^{\infty}f_{i}^{\beta}(x)dx},~i=1,2.
\end{eqnarray}
\end{theorem}
\begin{proof}
	After some calculations, it can be shown that 
	\begin{eqnarray}\label{eq4.7}
	r f_1^{\beta}(x)+(1-r)f_{2}^{\beta}(x)=\{r I_{\beta}(X_{1})+(1-r)I_{\beta}(X_{2})\}[f_{(\beta,1)}(x)\Gamma+f_{(\beta,2)}(x)(1-\Gamma)].
	\end{eqnarray}
	Now, from Definition \ref{def2.1}, using (\ref{eq4.4}) and (\ref{eq4.7}), we obtain 
	\begin{eqnarray*}
	I_{\beta}^{\omega}(M(r,\gamma))&=&\int_{0}^{\infty}\omega(x) f^{\beta}_{M(r,\gamma)}(x)dx\\
	&=&\int_{0}^{\infty}\omega(x)\frac{[r f_{1}^{\gamma}(x)+(1-r)f_{2}^{\gamma}(x)]^{\frac{\beta}{\gamma}}}{\left(\int_{0}^{\infty}[r f_{1}^{\gamma}(x)+(1-r)f_{2}^{\gamma}(x)]^{\frac{1}{\gamma}}dx\right)^{\beta}}dx	\\
	&=& \int_{0}^{\infty}\omega(x)\frac{(f_{(\beta,1)}(x)\Gamma+f_{(\beta,2)}(x)(1-\Gamma))^{\frac{\beta}{\gamma}}}{\left(\int_{0}^{\infty}(f_{(\beta,1)}(x)\Gamma+f_{(\beta,2)}(x)(1-\Gamma))^{\frac{1}{\gamma}}dx\right)^{\beta}}dx\\
	&=&\frac{ \int_{0}^{\infty}\omega(x)f_{\Gamma}^{\frac{\beta}{\gamma}}(x)dx}{\left(\int_{0}^{\infty}(f_{\Gamma}(x))^{\frac{1}{\gamma}}dx\right)^{\beta}}=\frac{I^{\omega}_{\frac{\beta}{\gamma}}(X_{\Gamma})}{\left(I_{\frac{1}{\gamma}}(X_{\Gamma})\right)^{\beta}}.
	\end{eqnarray*} 
.Thus, the result is established.
\end{proof}

In the following result, we study the GWRIGF between $(r,\gamma)$-mixture RV $M(r,\gamma)$ and the component random variables $X_{i}$, $i=1,2.$ The PDF of the escort RV of order $\frac{1}{\gamma}$ with baseline density $f_{\Gamma}(\cdot)$, denoted by $K(\frac{1}{\gamma},\Gamma)$ is  given by
\begin{eqnarray}\label{eq4.8}
f_{K(\frac{1}{\gamma},\Gamma)}(x)=\frac{f_{\Gamma}^{\frac{1}{\gamma}}(x)}{\int_{0}^{\infty}f_{\Gamma}^{\frac{1}{\gamma}}(x)dx},~~~x>0.
\end{eqnarray}
\begin{theorem}
	The GWRIGF between $M(r,\gamma)$ and $X_{i}$, $i=1,2$ is expressed as
	\begin{eqnarray}
	R_{\beta}^{\omega}(M(r,\gamma),X_{i})=R_{\beta}^{\omega}\left(K\left(\frac{1}{\gamma},\Gamma\right),X_{i}\right),
	\end{eqnarray}
	where the density of $K(\frac{1}{\gamma},\Gamma)$ is given in (\ref{eq4.8}).
\end{theorem}
\begin{proof}
	From Definition \ref{def3.1} and using the similar arguments of Theorem \ref{th4.3}, we obtain
	\begin{eqnarray*}
	R_{\beta}^{\omega}(M(r,\gamma),X_{i})&=&\int_{0}^{\infty}\omega(x)f^{\beta}_{M(r,\gamma)}(x)f_{i}^{1-\beta}(x)dx\\
	&=& \int_{0}^{\infty}\omega(x)f_{i}^{1-\beta}(x)\frac{[r f_{1}^{\gamma}(x)+(1-r)f_{2}^{\gamma}(x)]^{\frac{\beta}{\gamma}}}{\left(\int_{0}^{\infty}[r f_{1}^{\gamma}(x)+(1-r)f_{2}^{\gamma}(x)]^{\frac{1}{\gamma}}dx\right)^{\beta}}dx\\
	&=&\int_{0}^{\infty}\omega(x)\left(\frac{f^{\frac{1}{\gamma}}_{\Gamma}(x)}{\int_{0}^{\infty}f^{\frac{1}{\gamma}}_{\Gamma}(x)dx}\right)^{\beta}f_{i}^{1-\beta}(x)dx=R_{\beta}^{\omega}\left(K\left(\frac{1}{\gamma},\Gamma\right),X_{i}\right).
	\end{eqnarray*}
This completes the proof.
\end{proof}

{The informational energy is an idea adopted from the kinetic energy expression of Classical Mechanics. It measures the amount of uncertainty  or randomness of a random variable or a probability system 
like Shannon's entropy. It augments when randomness decreases which was first introduced and studied by \cite{onicescu1966informational}. The informational energy has also applications in other various fields correspond to the Herfindahl-Hirschman index to use indicators to detect anticompetitive behavior in industries (see \cite{matsumoto2012some}),  to the index coincidence (developed for cryptanalysis) (see \cite{harremoes2001inequalities}) in information theory and is associated with Simpson's diversity index (see \cite{simpson1949measurement}) in ecology.
 Suppose $X$ and $Y$ are two random variables. Then, the weighted $\beta$-cross informational energy is defined as 
\begin{align}\label{eq4.12}
CI^\omega_\beta(X,Y)=\int_{0}^{\infty}\omega(x)\sqrt{f^\beta(x)g^\beta(x)}dx,
\end{align}
where $\omega(x)>0$ is weight function.  From 
(\ref{eq4.12}), we obtain the following observations:
\begin{itemize}
\item $CI_{\beta}^{\omega}(X,Y)|_{\beta=1}=\int_{0}^{\infty}\omega(x)\sqrt{f(x)g(x)}dx$ is known as weighted Bhattacharyya coefficient (for details see \cite{du2012meanshift});
\item $CI_{\beta}^{\omega}(X,Y)|_{\beta=1, \omega=1}=\int_{0}^{\infty}\sqrt{f(x)g(x)}dx$ is called  Bhattacharyya coefficient (for details see \cite{kailath1967divergence});
\item$CI_{\beta}^{\omega}(X,Y)|_{\omega=1}=\int_{0}^{\infty}\sqrt{f^\beta(x)g^\beta(x)}dx$ is $\beta$-cross informational energy (see \cite{kharazmi2023optimal});
\item$CI_{\beta}^{\omega}(X,Y)|_{\omega=1, \beta=2}=\int_{0}^{\infty}f(x)g(x)dx$ is known as cross informational energy (see \cite{nielsen2022onicescu}).
\end{itemize}
Now, we establish the relation between weighted $\beta$-cross informational energy and GWIGF.
\begin{proposition}\label{prop4.1}
Suppose $X$ and $Y$ are two random variables and corresponding GWIGFs $I^\omega_\beta(X)$ and $I^\omega_\beta(Y)$, respectively. Then,
\begin{align*}
CI^\omega_\beta(X,Y)\le \frac{1}{2}[I^\omega_\beta(X)+I^\omega_\beta(Y)].
\end{align*}
\end{proposition}
\begin{proof}
  From geometric-arithmetic mean inequality, we have
\begin{align}\label{eq4.13}
\sqrt{f^\beta(x)g^\beta(x)}\le \frac{1}{2}[f^\beta(x)+g^\beta(x)].
\end{align}
From the definition of weighted $\beta$-cross informational energy in (\ref{eq4.12}) and using (\ref{eq4.13}), we obtain
\begin{align}\label{eq4.14}
CI^\omega_\beta(X,Y)&=\int_{0}^{\infty}\omega(x)\sqrt{f^\beta(x)g^\beta(x)}dx\nonumber\\
&\le\frac{1}{2}\int_{0}^{\infty}\omega(x)\big[f^\beta(x)+g^\beta(x)\big]dx.
\end{align}
Hence, the result follows from (\ref{eq4.14}). Therefore, the proof is completed.
\end{proof}
Next, we show that the GWIGF is closely related to the weighted $\beta$-cross informational energy for escort distributions.

\begin{proposition}
Suppose $X$ and $Y$ are two random variables. Then, the relation between the weighted $\beta$-cross informational energy between escort random variables $X_{e,\alpha}$ and $Y_{e,\alpha}$ and GWIGFs can be expressed as
\begin{align}
CI^\omega_\beta(X_{e,\alpha},Y_{e,\alpha})=\frac{CI^\omega_{\alpha \beta}(X,Y)}{\sqrt{(I_\alpha(X)I_\alpha(Y))^\beta}}\le \frac{I^\omega_{\alpha \beta}(X)+I^\omega_{\alpha \beta}(Y)}{2\sqrt{(I_\alpha(X)I_\alpha(Y))^\beta}},
\end{align}
where $I_\alpha(\cdot)$ is IGF and $f_{e, \alpha}(\cdot)$ and $g_{e,\alpha}(\cdot)$ are PDFs of $X_{e,\alpha}$ and $Y_{e,\alpha}$, respectively.
\end{proposition}
\begin{proof}
From (\ref{eq4.12}), we obtain
\begin{align}\label{eq4.16}
CI^\omega_\beta(X_{e,\alpha},Y_{e,\alpha})&=\int_{0}^{\infty}\omega(x)\sqrt{f^\beta_{e,\alpha}(x)g^\beta_{e,\alpha}(x)}dx\nonumber\\
&=\int_{0}^{\infty}\omega(x)\sqrt{\Big(\frac{f^\alpha(x)}{\int_{0}^{\infty}f^\alpha(x)dx}\Big)^\beta\Big(\frac{g^\alpha(x)}{\int_{0}^{\infty}g^\alpha(x)dx}\Big)^\beta}dx\nonumber\\
&=\frac{\int_{0}^{\infty}\omega(x)\sqrt{f^{\alpha \beta}(x)g^{\alpha \beta}(x)}dx}{\sqrt{(I_\alpha(X)I_\alpha(Y))^\beta}}.
\end{align}
From (\ref{eq4.16}), we have 
\begin{align}\label{eq4.17}
CI^\omega_\beta(X_{e,\alpha},Y_{e,\alpha})=\frac{CI^\omega_{\alpha \beta}(X,Y)}{\sqrt{(I_\alpha(X)I_\alpha(Y))^\beta}}.
\end{align}
Further, using Proposition \ref{prop4.1} in (\ref{eq4.17}), we obtain
\begin{align}\label{eq4.18}
CI^\omega_\beta(X_{e,\alpha},Y_{e,\alpha})\le \frac{I^\omega_{\alpha \beta}(X)+I^\omega_{\alpha \beta}(Y)}{2\sqrt{(I_\alpha(X)I_\alpha(Y))^\beta}}.
\end{align}
Combining (\ref{eq4.17}) and (\ref{eq4.18}), the required relation is made. Hence, completes the proof.
\end{proof}

}

\section{GWIGF and GWRIGF for residual lifetime}
The residual life is a useful concept in life testing studies. {It helps in determining maintenance when an asset might fail of need maintenance in engineering as this information is very needful for optimizing maintenance schedules, minimizing downtime and maximizing asset utilization. In finance, the residual lifetime uses to the planning knowing of investments which helps in making informed decisions about portfolio management and risk assessment. Residual life-based informational measures use for predictive maintenance decision-making in various fields like reliability engineering, medicine science and finance. In this regard, one may refer to \cite{do2022residual}.}
 Suppose a system has survived until  $t>0.$ Then, the additional life of the system is known as the residual lifetime. Denoting $X$ as the random lifetime of a system, the residual lifetime is the conditional RV, denoted by $X_{t}=[X-t|X>t]$. Its PDF is 
\begin{eqnarray}
f_{t}(x)=\frac{f(x)}{\bar{F}(t)},~~x>t>0.
\end{eqnarray}
This section focuses on the study of GWIGF and GWRIGF of the residual lifetime.  

\subsection{Residual GWIGF}
The following definition provides the mathematical representation of the residual GWIGF.
\begin{definition}
	Suppose $X$ has PDF $f(\cdot)$ and survival function $\bar{F}(\cdot)$.  The GWIGF for $X_{t}$ is given by 
	\begin{eqnarray}\label{eq5.2}
	I_{\beta}^{\omega}(X;t)=\int_{t}^{\infty}\omega(x)\left(\frac{f(x)}{\bar{F}(t)}\right)^{\beta}dx=\frac{1}{\bar{F}^{\beta-1}(t)}E[\omega(X)f^{\beta-1}(X)|X>t],~\beta\ge1.
	\end{eqnarray}
\end{definition}
Note that as $t\rightarrow 0$, the residual GWIGF in (\ref{eq5.2}) reduces to the GWIGF given in (\ref{eq2.1}). Further, (\ref{eq5.2}) becomes residual IGF due to \cite{kharazmi2021jensen}.  Suppose $X_{\beta}$ and $X$ are non-negative RVs with survival functions $\bar{F}_{\beta}(x)$ and $\bar{F}(x)$, respectively. Further, assume that $X_{\beta}$ and $X$ follow proportional hazards  (PH) model, that is, their survival functions satisfy 
\begin{eqnarray}\label{eq5.3}
\bar{F}_{\beta}(x)=(\bar{F}(x))^{\beta},~\beta>0.
\end{eqnarray}
We note that though the PH model was originally introduced by \cite{lehmann1953proportional}, it has been widely applied by researchers and practitioners after the rationale explained by \cite{cox1972regression} {in order
to estimate the effects of different covariates influencing the times to the
failures of a system. The model has been used rather extensively in
biomedicine, and recently, interest in its application in reliability engineering
has increased.} For some existing literature regarding PH model upto the year $1994$, we refer to \cite{kumar1994proportional}.  The hazard rate of $X_{\beta}$ is 
$h_{\beta}(x)=\beta\frac{f(x)}{\bar{F}(x)},~x>0$.
It can be established after some calculations that 
\begin{eqnarray}\label{eq5.4}
E[\omega(X_{\beta})h_{\beta}^{\beta-1}(X_{\beta})|X_{\beta}>t]=\beta^{\beta}\int_{t}^{\infty}\omega(x)\left(\frac{f(x)}{\bar{F}(t)}\right)^{\beta}dx=\beta^{\beta}I_{\beta}^{\omega}(X;t),
\end{eqnarray}
that is, the GWIGF for residual lifetime is expressed in the form of a conditional expectation of the hazard rate function of a PH model.
The $k$-th order derivative of the residual GWIGF given in (\ref{eq5.2}) with respect to $\beta$
is 
\begin{eqnarray}\label{eq5.5}
\frac{\partial^{k}I_{\beta}^{\omega}(X;t)}{\partial \beta^{k}}=\int_{0}^{\infty}\omega(x) \left(\frac{f(x)}{\bar{F}(t)}\right)^{\beta}\left(\log\frac{f(x)}{\bar{F}(t)}\right)^{k}dx,~\beta\ge1.
\end{eqnarray}
Now, from (\ref{eq5.2}) and (\ref{eq5.5}), we obtain the following observations:
\begin{itemize}
	\item $I_{\beta}^{\omega}(X;t)|_{\beta=1}=E(\omega(X)|X>t);$
	\item $I_{\beta}^{\omega}(X;t)|_{\beta=2}=-2 J^{\omega}(X;t);$
	\item $\frac{\partial I_{\beta}^{\omega}(X;t)}{\partial \beta}|_{\beta=1}=-H^{\omega}(X;t),$
\end{itemize}
where $J^{\omega}(X;t)=-\frac{1}{2}\int_{t}^{\infty}\omega(u)(\frac{f(u)}{\bar{F}(t)})^{2}du$
is known as the residual weighted extropy (see \cite{sathar2021dynamic}) and $H^{\omega}(X;t)=-\int_{t}^{\infty}\omega(u)\frac{f(u)}{\bar{F}(t)}\log \frac{f(u)}{\bar{F}(t)}du$ is known as the residual weighted Shannon entropy. {In particular, for $
\beta=2$, the residual GWIGF in (\ref{eq5.2}) reduces to $\int_{t}^{\infty}\omega(x)\left(\frac{f(x)}{\bar{F}(t)}\right)^{2}dx$,  which is known as residual weighted informational energy function. } The following example provides closed-form expression of the residual GWIGF for Pareto and exponential distributions. Note that for the computation purpose, here we take $\omega(x)=x$.
\begin{example}\label{ex5.1}~
	\begin{itemize}
	\item[(i)] If $X$ follows Pareto-I type distribution with PDF $f(x)=\frac{a \gamma^a}{x^{a+1}},$ for $x>\gamma$ and $a>0$, then
	\begin{eqnarray}\label{eq5.6}
	I^\omega_\beta(X;t)=\frac{1}{a-1}\left(\frac{a \gamma^a}{1-(\frac{\gamma}{t})^a}\right)^\beta t^{1-a},~a>1,~\beta\ge1;
	\end{eqnarray}
	The graph of (\ref{eq5.6}) with respect to $t$ is depicted in Figure $4(a)$.
	\item[(ii)] Suppose $X$ has CDF $F(x)=1-e^{-\lambda x},~ x>0$ and $\lambda>0$. Then, 
	\begin{eqnarray}\label{eq5.7}
	I^\omega_\beta(X;t)=\frac{\lambda^\beta}{\beta^2\lambda^2}(\beta \lambda t+1),~\beta\ge1.
	\end{eqnarray}
	The graph of (\ref{eq5.7}) with respect to $t$ is presented in Figure $4(b).$
	\end{itemize}
	\begin{figure}[h!]\label{fig5}
		\centering
	\subfigure[]{\label{c1}\includegraphics[height=1.9in]{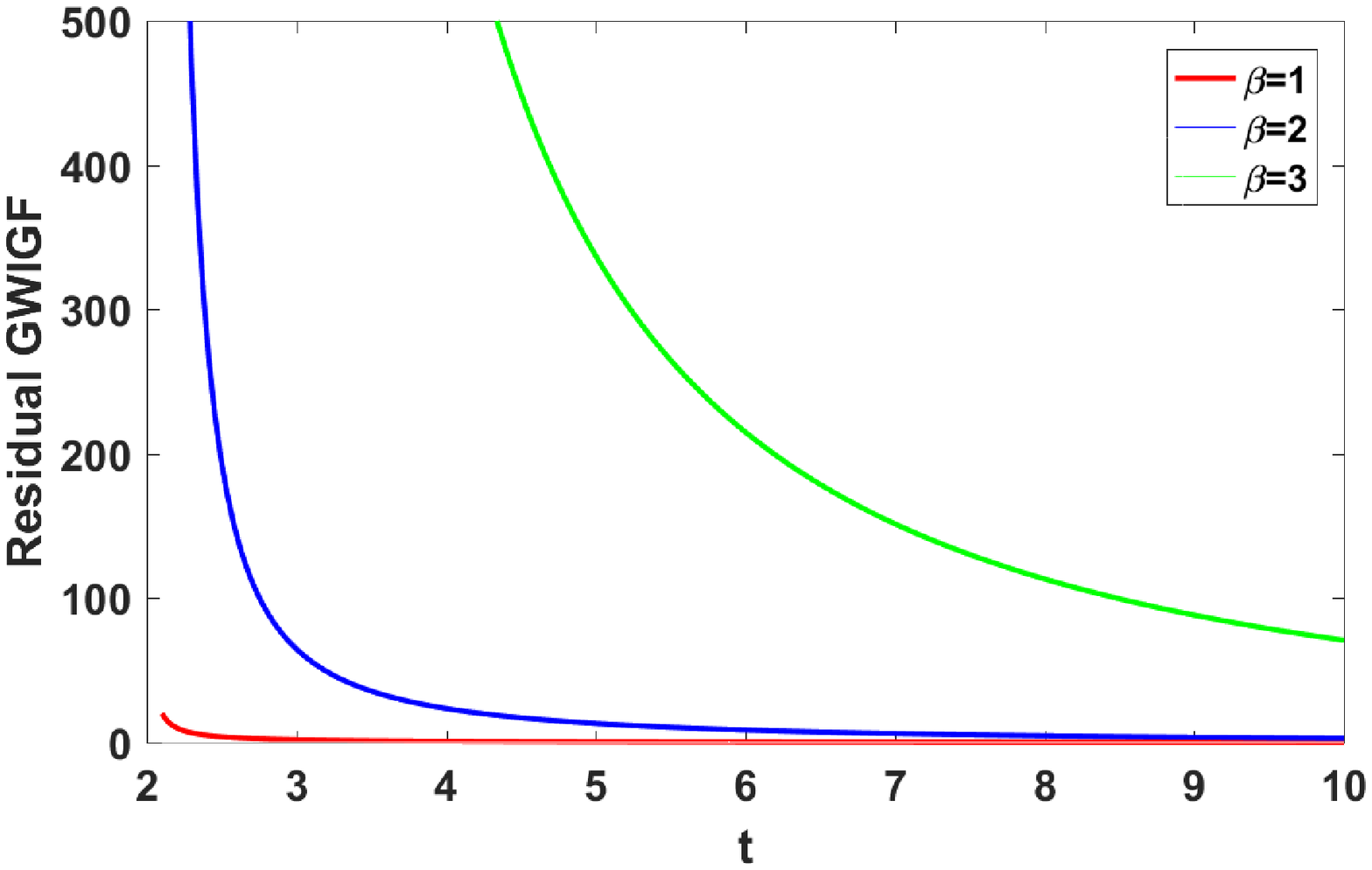}}
	\subfigure[]{\label{c1}\includegraphics[height=1.9in]{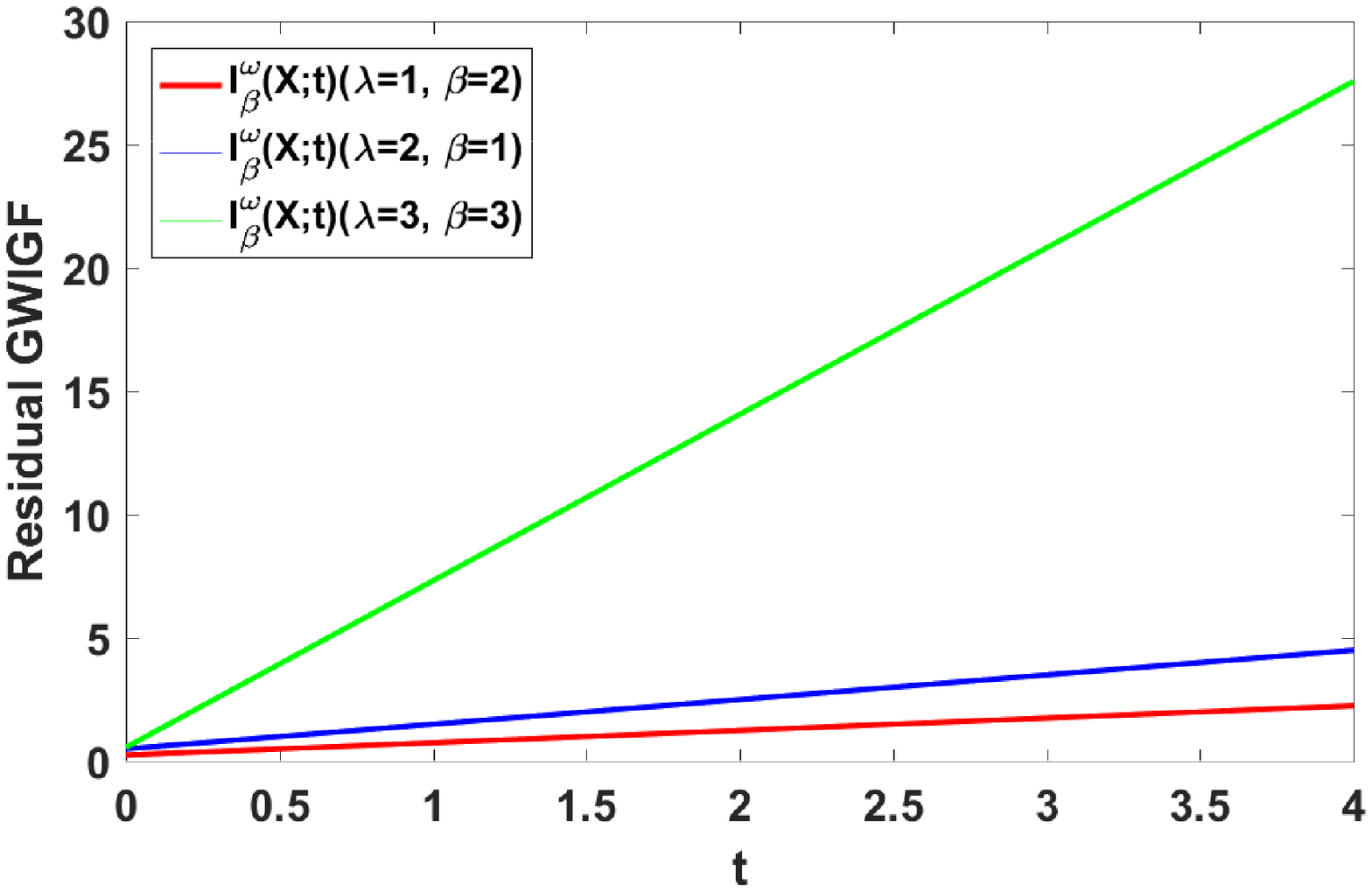}}
		\caption{$(a)$ Graph for residual GWIGF when $\gamma=2$, $a=3$  as in Example \ref{ex5.1}$(i)$. Here, three different choices of $\beta$, say $1,2,$ and $3$ are considered.
		$(b)$ Plot of the residual GWIGF  as in Example \ref{ex5.1}$(ii)$.}
	\end{figure}
\end{example}

Denote by $h(x)=\frac{f(x)}{\bar{F}(x)}$ the hazard rate of $X$ with PDF $f(\cdot)$ and survival function  $\bar{F}(\cdot)$. The cumulative hazard rate of $X$, $\widetilde{H}(x)=-\log \bar{F}(x)$ can be expressed in terms of the hazard rate as 
\begin{eqnarray}
\widetilde{H}(x)=\int_{0}^{x}h(u)du,
\end{eqnarray}
 {Note that it plays an important role in the study of ageing properties of lifetime variables.
Next, we establish a relation between the expectation of GWIGF with respect to proportional hazard rate random variable $X_\beta$ and the expectation of cumulative hazard rate of $X$.   }

\begin{theorem}\label{th5.1}
	We have 
	\begin{eqnarray}
	E[\omega(X)f^{\beta-1}(X)\widetilde{H}(X)]=\frac{1}{\beta}E_{X_{\beta}}\left[I_{\beta}^{\omega}(X;X_{\beta})\right], 
	\end{eqnarray}
	where $E_{X_{\beta}}[\cdot,\cdot]$ means the expectation is with respect to  $X_{\beta}$.
\end{theorem}

\begin{proof}
	We obtain
	\begin{eqnarray*}
	E[\omega(X)f^{\beta-1}(X)\widetilde{H}(X)]&=&\int_{0}^{\infty}\omega(x)f^{\beta}(x)\left(\int_{0}^{x}h(t)dt\right)dx\\
	&=& \int_{0}^{\infty}h(t)\left(\int_{t}^{\infty}\omega(x)f^{\beta}(x)dx\right)dt\\
	&=& \int_{0}^{\infty}f(t)\left(\int_{t}^{\infty}\omega(x)f^{\beta-1}(x)\frac{f(x)}{\bar{F}(t)}dx\right)dt\\
	&=&\int_{0}^{\infty}f(t) E[\omega(X)f^{\beta-1}(X)|X>t]dt\\
	&=&\int_{0}^{\infty}f(t)\bar{F}^{\beta-1}(t)I_{\beta}^{\omega}(X;t)dt\\
	&=& \frac{1}{\beta}E_{X_{\beta}}\left[I_{\beta}^{\omega}(X;X_{\beta})\right],
	\end{eqnarray*}
where the second equality is due to Fubinni's theorem and the fifth equality is due to  (\ref{eq5.2}). The proof is established.
\end{proof}
{We note that when case $\omega(x)=1$, Theorem  \ref{th5.1} reduces to Theorem $15$ of \cite{kharazmi2021jensen}.} In the renewal theory, the equilibrium distribution arises naturally. It is a special case of the weighted distribution {to determine the mean value of thermodynamic quantities and also plays an important role in reliability studies, for details see \cite{gupta2007role}}. Associated with an RV $X$, the CDF and PDF of the equilibrium RV $X_{E}$ are respectively given by 
\begin{eqnarray}
F_{e}(x)=\frac{1}{\mu}\int_{0}^{x}\bar{F}(t)dt~~\mbox{and}~~f_{e}(x)=\frac{\bar{F}(x)}{\mu},~~x>0,
\end{eqnarray}
where $E(X)=\mu<\infty.$ For details about equilibrium distribution, we refer to \cite{fagiuoli1993new}. The mean residual life (MRL) is a useful concept in reliability and survival analysis, see \cite{vu2015stationary}. Given that a system has survived upto time $t$, the MRL provides average remaining lifetime of the system. For a system with lifetime $X$, the MRL is given by
\begin{eqnarray}\label{eq5.11}
M_{X}(t)=E[X-t|X>t]=\int_{t}^{\infty}\frac{\bar{F}(x)}{\bar{F}(t)}dx.
\end{eqnarray}
In a similar way, the MRL of $X_{\beta}$ can be defined and we denote it by $M_{X_{\beta}}(t).$ The next theorem, {we evaluate residual GWIGF of equilibrium distribution and } present a relation between the residual GWIGF of $X_{E}$ with the MRL of $X$ and weighted MRL of $X_{\beta}$.

\begin{theorem}\label{th5.2}
	The residual GWIGF of $X_{E}$ is 
	\begin{eqnarray}
	I_{\beta}^{\omega}(X_{E};t)=\frac{M_{X_{\beta}}^{\omega}(t)}{(M_{X}(t))^{\beta}},~t>0,
	\end{eqnarray}
	where $M_{X}(t)$ is given in (\ref{eq5.11}) and $M_{X_{\beta}}^{\omega}(t)$ is the weighted MRL of $X_{\beta}$ with survival function in (\ref{eq5.3}).
\end{theorem}

\begin{proof}
	We have 
	\begin{eqnarray}
	I_{\beta}^{\omega}(X_{E};t)=\int_{0}^{\infty}\omega(x)\left(\frac{f_{e}(x)}{\bar{F}_{e}(t)}\right)^{\beta}dx=\frac{\int_{0}^{\infty}\omega(x)\frac{\bar{F}^{\beta}(x)}{\bar{F}^{\beta}(t)}dx}{\left(\int_{t}^{\infty}\frac{\bar{F}(x)}{\bar{F}(t)}dx\right)^{\beta}}=\frac{M_{X_{\beta}}^{\omega}(t)}{(M_{X}(t))^{\beta}}.
	\end{eqnarray}
	Thus, the result is proved.
\end{proof}
{From Theorem \ref{th5.2},  we observe that for different weight functions produce different residual GWIGF of equilibrium distribution. In particular, the weight function $\omega(x)=1$, Theorem \ref{th5.2} reduces to  Theorem $14$ of \cite{kharazmi2021jensen}. The hazard rate order provides valuable insights into the comparative behaviour of failure rates or hazard rates of different systems or populations. For example, in medical research and epidemiology, hazard rate order is used to compare the mortality or recurrence rates among different patient groups or treatment interventions. This information aids healthcare professionals in making informed decisions about patient care and treatment strategies.} Next, the sufficient conditions for which the residual GWIGFs of two distributions are ordered, have been derived.
\begin{theorem}
	Suppose the hazard rates of $X$ and $Y$ are denoted by $h(\cdot)$ and $s(\cdot)$, respectively. Further, suppose either $F(\cdot)$ or $G(\cdot)$ has DFR, and  $Y \leq_{hr}X$. Then,  $$I^\omega_\beta(X;t)\leq I^\omega_\beta(Y;t), ~~\text{for all}~ \beta\geq1,$$
	provided $\omega(\cdot)$ is decreasing. 
\end{theorem}

\begin{proof}
	It is known that $Y\leq_{hr}X$ implies $s(x)\geq h(x),$ for  $x>0$. Further,  $Y\leq_{hr}X\implies Y_t\leq_{st}X_t\implies\frac{\bar F(x)}{\bar F(t)}\geq\frac{\bar G(x)}{\bar G(t)},$ for all $x\geq t>0$ (see \cite{shaked2007stochastic}). Thus, for $\beta\ge1, 
	$$$X_{\beta,t}=[X_\beta-t|X_\beta>t]
	\geq_{st}[Y_\beta-t|Y_\beta>t]=Y_{\beta,t},$$
	where $X_\beta$ and $Y_\beta$ are PH RVs associated with $X$ and $Y,$ respectively. Under the assumption made, $F(\cdot)$ is DFR and $\omega(\cdot)$ is decreasing. This implies that $\omega(x)h_{\beta}^{\beta-1}(x)$ is decreasing with respect to $x$, for $\beta\ge1$. Now, using (\ref{eq5.4}) and $(1.A.7)$ of \cite{shaked2007stochastic}, we get
	\begin{eqnarray*}
		I^\omega_\beta(X;t)=\frac{E[\omega(X_{\beta})h^{\beta-1}_{\beta}(X_{\beta})|X_\beta>t]}{\beta^{\beta}}
		&\leq&\frac{E[\omega(Y_{\beta})h^{\beta-1}_{\beta}(Y_{\beta})|Y_\beta>t]}{\beta^{\beta}}\\&\le& \frac{E[\omega(Y_{\beta})s^{\beta-1}_{\beta}(Y_{\beta})|Y_\beta>t]}{\beta^{\beta}}
		=I^\omega_\beta(Y;t).
	\end{eqnarray*} 
	Thus, the theorem is established.
\end{proof}  

{It is well known that the monotone functions play a crucial role in information theory, a branch of mathematics and computer science that deals with quantifying and analyzing information.  In data compression algorithms, monotone functions are applied to transform data in a way that reduces redundancy and minimizes the amount of information needed to represent it.} Now, bounds of the residual GWIGF are obtained under the assumption that it is monotone with respect to $t.$
\begin{theorem}\label{th5.4}
	Suppose $X$ is an RV with CDF and PDF $F(\cdot)$ and $f(\cdot),$ respectively. Further, let $I_{\beta}^{\omega}(X;t)$ be increasing (decreasing) in $t$. Then,
	\begin{eqnarray}\label{eq5.14}
	I^\omega_\beta(X;t)\geq (\leq)\frac{1}{\beta}\omega(t)h^{\beta-1}(t),~ \text{for}~ \beta\ge1.
	\end{eqnarray}
\end{theorem}

\begin{proof}
	Differentiating (\ref{eq5.2}) with respect to $t$
	\begin{eqnarray}\label{eq5.15}
	\frac{d}{dt} I^\omega_\beta(X;t)=-\omega(t)h^\beta(t)+\beta h(t) I^\omega_\beta(X;t).
	\end{eqnarray}
	Under the assumptions made,  from (\ref{eq5.15}), we obtain 
	\begin{eqnarray}\label{eq5.16}
	-\omega(t)h^\beta(t)+\beta h(t)I^\omega_\beta(X;t) \geq (\leq)0,
	\end{eqnarray}
	from which the result readily follows. 
\end{proof}

For the purpose of illustration of the result in Theorem \ref{th5.4}, the following example is presented.

\begin{example}\label{ex5.2}
	Suppose $X$ follows exponential distribution with rate parameter $\lambda>0.$ From (\ref{eq5.7}), clearly $I_{\beta}^{\omega}(X;t)$ is increasing with respect to $t$. Thus, from  Theorem \ref{th5.4}, we have $I^\omega_\beta(X;t)\geq\frac{1}{\beta}t\lambda^{\beta-1},$ for all $t>0,$ which can be easily verified from Figure $5.$ 
\end{example}

 \begin{figure}[h!]
	\centering
	\includegraphics[width=14cm,height=8cm]{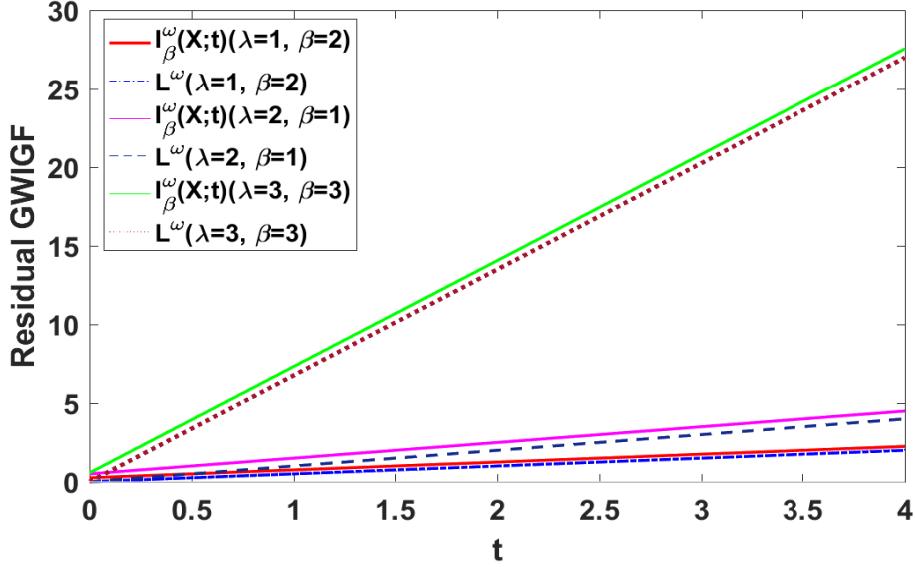}
	\caption{Graphs of the GWIGF for residual lifetime and $L^\omega(\lambda,\beta)=\frac{1}{\beta}t\lambda^{\beta-1}$ as in Example \ref{ex5.2}.}
\end{figure}

\subsection{Residual GWRIGF}
First, we will introduce the GWRIGF for residual lifetime.
\begin{definition}
	Suppose $X$ and $Y$ are  two RVs with survival functions $\bar{F}(\cdot)$ and $\bar{G}(\cdot),$ and PDFs $f(\cdot)$ and $g(\cdot),$ respectively. Then, for any $\beta\ge1$, the residual GWRIGF between $X$ and $Y$ is 
	\begin{eqnarray}\label{eq5.17}
	R^\omega_\beta(X,Y;t)=\int^\infty_t\omega(x)\left(\frac{f(x)}{\bar F(t)}\right)^\beta\left(\frac{g(x)}{\bar G(t)}\right)^{1-\beta}dx,
	\end{eqnarray}
	where $\omega(x)$ is non-negative real valued measurable function.
\end{definition}

When $\omega(x)=1$, the residual GWRIGF becomes  residual RIGF. The following example produces closed form expression of the residual GWRIGF. 
\begin{example}
	Suppose $X$ and $Y$ follow Pareto-I distributions with CDFs $F(x)=1-x^{-c}$ and $G(x)=1-x^{-\gamma}$, $c,~\gamma>0$ and $x\geq1$ and PDFs $f(x)= cx^{-(1+c)}$ and $g(x)=\gamma x^{-(1+\gamma)},$ respectively. Then, for $\beta\ge1$ we obtain
	$$ R^\omega_\beta(X,Y;t)=\frac{c^\beta \gamma^{1-\beta}}{1-\gamma+\beta c-\beta \gamma}t,~~\beta c-\beta \gamma+\gamma-1>0.$$
\end{example}

Differentiating (\ref{eq5.17}) with respect to $\beta$ we get 
\begin{eqnarray}
\frac{\partial}{\partial \beta}R_{\beta}^{\omega}(X,Y;t)=\int_{t}^{\infty}\omega(x) \frac{g(x)}{\bar{G}(t)} \left(\frac{f(x)/\bar{F}(t)}{g(x)/\bar{G}(t)}\right)^{\beta}\log\left(\frac{f(x)/\bar{F}(t)}{g(x)/\bar{G}(t)}\right)dx.
\end{eqnarray}

The following observations can be easily pointed out:
\begin{itemize}
	\item $R^\omega_\beta(X,Y;t)|_{\beta=1}=E[\omega(X)|X>t];$
	\item $R^\omega_\beta(X,Y;t)|_{t=0}=R^\omega_\beta(X,Y);$
	\item $R^\omega_\beta(X,Y;t)=R^{\omega}_{1-\beta}(Y,X;t);$
	\item $\frac{\partial}{\partial \beta}R_{\beta}^{\omega}(X,Y;t)|_{\beta=0}=KL^{\omega}(X,Y;t)$,
\end{itemize}
where $KL^{\omega}(X,Y;t)=\int_{t}^{\infty}\omega(x)\frac{f(x)}{\bar{F}(t)}\log\left(\frac{f(x)/\bar{F}(t)}{g(x)/\bar{G}(t)}\right)dx$ is called as the weighted residual KL divergence (see \cite{moharana2019weighted}). \\

Following theorem presents the effects of the residual GWRIGF under monotonic transformations.
\begin{theorem}
	Suppose $X$ and $Y$ are RVs with survival functions $\bar{F}(\cdot)$ and $\bar{G}(\cdot)$, and PDFs $f(\cdot)$ and $g(\cdot)$, respectively. Let $\psi(\cdot)$ be a strictly monotonic, differentiable and invertible function. Then,
	\begin{equation}
	R^\omega_\beta(\psi(X),\psi(Y);t)=\left\{
	\begin{array}{ll}
	R^{\omega_\psi}_\beta\big(X,Y;\psi^{-1}(t)\big),~\text{if  $\psi$  is  strictly  increasing;}
	\\
	\\
	-R^{\omega_\psi}_\beta\big(X,Y;\psi^{-1}(t)\big),~\text{if  $\psi$  is  strictly  decreasing,}
	\end{array}
	\right.
	\end{equation}
	where $\omega_\psi(x)=\omega(\psi(x))$ is a non-negative real valued measurable function.
\end{theorem}

\begin{proof}
	We omit the proof since it is analogous to Theorem \ref{th3.1}. 
\end{proof}

\section{Non-parametric estimator for residual GWIGF }
In this section, we propose a non-parametric estimator of the residual GWIGF in (\ref{eq5.2}) based on the kernel density estimator of $f(\cdot)$, given by
\begin{eqnarray}\label{eq6.1}
\widehat f(x_i)=\frac{1}{nb_n}\sum_{i=1}^{n}k\left(\frac{x-X_i}{b_n}\right), 
\end{eqnarray}
where $k(\cdot)~(\ge0)$ is known as kernel, satisfying Lipschitz condition with $\int k(x)dx=1$. Further, $k(\cdot)$ is symmetric with respect to the origin. Here, $\{b_n\}$, known as bandwidths is a sequence of positive real numbers such that $b_n\rightarrow0$ and $nb_n\rightarrow\infty$ for $n\rightarrow\infty$. For details about kernel density estimator, plese see the references by \cite{rosenblatt1956remarks} and \cite{parzen1962estimation}. The non-parametric kernel estimator of $I^w_\beta(X;t)$ is 
\begin{eqnarray}\label{eq6.2}
\widehat I^w_\beta(X;t)=\int_{t}^{\infty}x\left(\frac{\widehat f(x)}{\widehat{ \bar {F}}(t)}\right)^\beta dx,~\beta\geq1,
\end{eqnarray}
where $\widehat{ \bar {F}}(t)=\int_{t}^{\infty}\widehat f(x)dx$. Next, we consider a simulated data set and two real data sets for illustrating the proposed non-parametric estimator in (\ref{eq6.2}).

\subsubsection*{Simulated data set}
We have generated a data set from the exponential distribution using Monte Carlo simulation technique. The simulation has been performed in ``Mathematica" software. The true value of the parameter {(here scale parameter)} of exponential distribution is considered as $\lambda=0.5.$ For the purpose of estimation, the Gaussian kernel has been employed, which is given by $$k(x)=\frac{1}{\sqrt{2\pi}}e^{-\frac{x^2}{2}},~x\in\mathbb{R}.$$
{Using $600$ bootstrap samples, the values of the} bias and mean squared error (MSE) of the proposed estimator have been computed for different values of $t$, $\beta$ and $n$, which are presented in Table \ref{tb1}. From Table \ref{tb1}, we notice that the MSE of the proposed estimator decreases as the sample size $n$ increases, which gurantees the consistency of the proposed estimator. {In addition to this observation, we also observe that the bias of the estimator is getting improved when the sample size increases.} Similar observation is noticed when $\beta$ increases. 

\begin{table}[h!]
	\centering 
		\caption {The bias and MSE (in bracket) for the kernel estimator of residual GWIGF in (\ref{eq6.2}).}
	\scalebox{1.1}{\begin{tabular}{c c c c c c c c } 
			\hline\hline\vspace{.1cm} 
			$\beta$ &$t$ & $n=30$ &$n=50$ &$n=70$ &$n=100$\\
			\hline\hline
			
			\multirow{10}{1.4cm}{1.2} & 0.1 & 0.07373 & 0.05863 & 0.05629& 0.04864  \\
			~ & ~ &  (0.04178)& (0.02437) & (0.01886) & (0.01305)  \\[1.2ex]
			~ & 0.2 &0.08917 & 0.06953 & 0.5834 & 0.05266   \\
			~ & ~ &  (0.04717)& (0.02838) & (0.01998) & (0.01390) \\[1.2ex]
			~ & 0.5 & 0.13204 & 0.10726 & 0.10199 & 0.09089 \\
			~ & ~ &(0.06276) & (0.04156) & (0.03017) & (0.02212)  \\[1.2ex]
			~ & 0.7 & 0.18795 & 0.15415 & 0.13532 & 0.13094  \\
			~ & ~ & (0.09471) & (0.05791) & (0.04032) & (0.03301) \\[1.2ex]	
			~ & 0.9 & 0.22829 & 0.19873 &  0.17720 & 0.15998  \\
			~ & ~ & (0.12825) & (0.08539) & (0.05660) & (0.04604) \\[1ex]		
			\hline
			\multirow{10}{1.4cm}{1.7} & 0.1 & -0.00382 & -0.00539 & -0.00370& -0.00484  \\
			~ & ~ &  (0.00314)& (0.00226) & (0.00134) & (0.00112)  \\[1.2ex]
			~ & 0.2 & 0.00598 & -0.00296 & -0.00203 &-0.00291   \\
			~ & ~ &  (0.00397)& (0.00214) & (0.00143) & (0.00103) \\[1.2ex]
			~ & 0.5 & 0.03349 & 0.02178 & 0.02115 & 0.02235\\
			~ & ~ &(0.00716) & (0.00377) & (0.00277) & (0.00216)  \\[1.2ex]
			~ & 0.7 & 0.06578 & 0.05669 & 0.0504 & 0.04628  \\
			~ & ~ & (0.01405) & (0.00915) & (0.00618) & (0.00441) \\[1.2ex]	
			~ & 0.9 & 0.10843 & 0.08372 &  0.08253 & 0.07401 \\
			~ & ~ & (0.02835) & (0.01516) & (0.01253) & (0.00888) \\[1ex]		
			\hline
			\multirow{10}{1.4cm}{2.5} & 0.1 & -0.02357 & -0.02232 & -0.02061& -0.01870 \\
			~ & ~ &  (0.00088)& (0.00070) & (0.00059) & (0.00047)  \\[1.2ex]
			~ & 0.2 & -0.02413 & -0.02178 & -0.02109 &-0.01893   \\
			~ & ~ &  (0.00097)& (0.00078) & (0.00065) & (0.00050) \\[1.2ex]
			~ & 0.5 &-0.01471 & -0.01397 & -0.01189 & -0.00982\\
			~ & ~ &(0.00161) & (0.00101) & (0.00068) & (0.00052)  \\[1.2ex]
			~ & 0.7 & 0.00357 & 0.00272 & 0.00210 & 0.00201  \\
			~ & ~ & (0.00334) & (0.00158) & (0.00108) & (0.00073) \\[1.2ex]	
			~ & 0.9 & 0.02581 & 0.02079 &  0.02012 & 0.01998 \\
			~ & ~ & (0.00683) & (0.00347) & (0.00268) & (0.00181) \\[1ex]		
			\hline\hline	 		
	\end{tabular}} 

	\label{tb1} 
\end{table}

\subsubsection*{Real data sets}

Here, we consider two data sets representing the remission times (in months) of $128$ bladder cancer data (see \cite{lee2003statistical}) and relief times of $20$ patients who received an analgesic (see \cite{gross1975survival}). To compute the bias and MSE of the proposed estimator, the Gaussian kernel as in simulated data set has been used. {The data sets are provided below.} We note that these data sets have been used by \cite{maiti2023progressively} and \cite{dutta2023statistical} for their study. 

\begin{table}[h!]
	\caption {The bladder cancer data set (data set-I).}
	\centering 
	\scalebox{1.1}{\begin{tabular}{c c c c c c c c } 
			\hline 
			0.08, 0.20, 0.40, 0.50, 0.51, 0.81, 0.90, 1.05, 1.19, 1.26, 1.35,1.40, 1.46, 1.76,\\[0.5Ex]
			2.02, 2.02, 2.07, 2.09, 2.23, 2.26, 2.46, 2.54, 2.62, 2.64, 2.69, 2.69, 2.75, 2.83,\\[0.5Ex]
			2.87, 3.02, 3.25, 3.31, 3.36, 3.36, 3.48, 3.52, 3.57, 3.64, 3.70, 3.82, 3.88, 4.18,\\[0.5Ex]
			4.23, 4.26, 4.33, 4.34, 4.40, 4.50, 4.51, 4.87, 4.98, 5.06, 5.09, 5.17, 5.32, 5.32, \\[0.5Ex]
			
			5.34, 5.41, 5.41, 5.49, 5.62, 5.71, 5.85, 6.25, 6.54, 6.76, 	6.93, 6.94, 6.97, 7.09,\\[0.5Ex]
			7.26, 7.28, 7.32, 7.39, 7.59, 7.62, 7.63,  7.66, 7.87, 7.93, 8.26, 8.37, 8.53, 8.65,\\[0.5Ex]
			8.66,~ 9.02, 9.22, 9.47, 9.74, 10.06, 10.34, 10.66, 10.75, 11.25, 11.64 , 11.79, \\[0.5Ex]
			11.98, 12.02, 12.03, 12.07, 12.63, 13.11, 13.29, 13.80, 14.24, 14.76, 14.77,  14.83,\\[0.5Ex]
			15.96, 16.62, 17.12, 17.14, 17.36, 18.10, 19.13, 20.28, 21.73, 22.69, 23.63,  25.74, \\[0.5Ex]
			25.82, 26.31, 32.15, 34.26, 36.66, 43.01, 46.12, 79.05.\\
			
			\hline
	\end{tabular}} 	
	\label{tb1} 
\end{table}

\begin{table}[h!]
	\caption {The relief times data set (data set-II).}
	\centering 
	\scalebox{1.1}{\begin{tabular}{c c c c c c c c } 
			\hline 
			1.1,~ 1.2,~ 1.3,~ 1.4,~ 1.4,~ 1.5, ~1.6,~ 1.6, ~1.7,~ 1.7,\\[0.5Ex] 1.7,~ 1.8,~ 1.8, ~1.9, ~2.0, ~2.2, ~2.3, ~2.7,~ 3.0,~ 4.1.\\
			\hline
	\end{tabular}} 	
	\label{tb1} 
\end{table} 

{In order to check the model which fits better than other models, here we have employed goodness of fit test. In this purpose, we use several techniques for checking the goodness of fit test, such as the negative
log-likelihood criterion (-ln L), Akaike’s-information criterion (AIC), AICc, and the Bayesian information criterion (BIC). According to the goodness of fit test, generalised X-exponential (GXE) distribution fits better than  the three parameter bathtub-shaped (Tbathtub) and inverted exponentiated half logistic (IEHL) distributions. The values of MLEs and four goodness of fit test statistics are given in Table $4$. From Table $4$, we observe that the values of the test statistics of GXE distribution are smaller than the other two distributions. Therefore, the GXE distribution is taken as a fitted model for the provided data set given in Table $2$.} The values of bias and MSE of (\ref{eq6.2}) have been evaluated using $600$ bootstrap samples of size $n=128$, which are presented in Table \ref{tb2} for different $t$ and $\beta$ values. We have only considered the value of $b_n$ as $0.20$. However, one may consider other values of $b_n$ for computational purposes.
	\begin{table}[h]
		\centering
		\caption {{The MLE, BIC, AICc, AIC, and negative log-likelihood values of the statistical models for the real data set-I}}
		\scalebox{1}{\begin{tabular}{ccccccc}
				\toprule
				\textbf{Model}   & \textbf{Shape}  & \textbf{Scale}  & \textbf{-ln L}  & \textbf{AIC}& \textbf{AICc} & \textbf{BIC}  \\
				\midrule
				GXE   & $\widehat{\alpha}=0.31806$ & $\widehat{\lambda}=0.00398$ & 468.6807& 941.3613 & 941.4573 & 947.0654 \\[1.2ex]
				Tbathtub  & $\widehat{\alpha}= 1.36667$ & $\widehat{\gamma}=1.13333$  &  1060.126 & 2126.252 &  2126.454 &  2134.689\\
				~ &  ~ & $\widehat{\beta}= 0.33333$  & ~ & ~&~&~ \\[1.2ex]
				IEHL &  $\widehat{\alpha}=0.58251$ & $\widehat{\lambda}=0.44029$  & 471.5664 & 947.1328 & 947.2288& 952.8368 \\
				\bottomrule
		\end{tabular}}
	\end{table}
\begin{table}[h]
	\centering
	\caption {{The MLE, BIC, AICc, AIC, and negative log-likelihood values of the statistical models for the real data set-II}}
	\scalebox{1}{\begin{tabular}{ccccccc}
			\toprule
			\textbf{Model}  & \textbf{Shape}  & \textbf{Scale}  & \textbf{-ln L}  & \textbf{AIC}& \textbf{AICc} & \textbf{BIC}  \\
			\midrule
			Gumbel-II  & $\widehat{\alpha}= 4.0172$ & $\widehat{\lambda}=6.0221$  &  15.4089 & 34.8174 &  35.5233 &  36.8089\\[1.2ex]
			GXE &  $\widehat{\alpha}=3.9618$ & $\widehat{\lambda}=0.5388$ & 19.3554 & 43.4167&   42.7108 & 44.7023 \\[1.2ex]
			EXP &  $\widehat{\lambda}=0.5263$ & ~  & 32.8371 & 67.6742 & 67.8964& 68.6699 \\[1.2ex]
			IEHL & $\widehat{\alpha}=11.6348$ & $\widehat{\lambda}=0.1582$  & 16.9492 & 37.8984 & 38.6043& 39.8899 \\
			\bottomrule
	\end{tabular}}
\end{table}

{Further, we have considered the data set on the relief times of the patients, received an analgesic, given by Table $3$. For checking the fitted model corresponding to the data set, we have explored goodness of fit test similar to the data set-I given in Table $2.$ We have compared four statistical models, say Gumbel-II, GXE, exponential and IEHL distributions. Based on the negative log-likelihood criterion, Akaike’s-information criterion, AICc, and the Bayesian information criterion, we conclude that this data set fits Gumbel Type-II distribution  with shape parameter $\alpha=4.0172$ and scale parameter $\lambda=6.0221$.} For computing bias and MSE of the proposed non-parametric estimator, $1000$ bootstrap samples of size $n=20$  have been considered. Various values of $t$ and $\beta$ with $b_n=0.56$ are taken. The bias and MSEs are presented in Table \ref{tb3}.  From Tables \ref{tb2} and \ref{tb3}, it is clear that {the estimator has small bias and less MSE. Thus, we can conclude that}  the proposed non-parametric estimator performs well.

\begin{table}[ht]
	\centering 
		\caption {The bias and MSE (in bracket) for the kernel estimator of residual GWIGF in (\ref{eq6.2}) based on bladder cancer data set of $128$ patients.}
	\scalebox{1.0}{\begin{tabular}{c c c c c c c c } 
			\hline\hline\vspace{.1cm} 
			$t$ &$ \beta=1.5$ & $\beta=2.0$ &$\beta=2.5$ &$\beta=3.0$ &$\beta=4.0$\\
			\hline\hline
			
			0.1 & -0.16868 & 0.11325 & 0.06698& 0.02737 &0.00395 \\
			~ &  (0.05287)& (0.01456) & (0.00479) & (0.00079)&  (0.000017)\\[1.2ex]
			0.3 & -0.25999 & 0.09549 & 0.06460 & 0.02792 & 0.00407  \\
			~ &  (0.08796)& (0.01087) & (0.00448) & (0.00084)& (0.000018)\\[1.2ex]
			0.5 & -0.28041 & 0.09504 & 0.06529 & 0.02916 & 0.00446\\
			~ &(0.09929) & (0.01075) & (0.00457) & (0.00091)& ( 0.000023)\\[1.2ex]
			0.7 & -0.34453 & 0.08737 & 0.06525& 0.02906&  0.00445\\
			~ & (0.13840) & (0.00952) & (0.00457) & (0.00091)& (0.000022)\\[1.2ex]	
			1.0 & -0.35949 & 0.08619 &  0.06655& 0.03115 & 0.00459\\
			~ & (0.14986) & (0.00916) & (0.00476) & (0.00104) & (0.0000236)\\[1.2ex]
			1.5 & -0.31652 & 0.11238 &  0.08034& 0.03617& 0.00586 \\
			~ & (0.12339) & (0.01458) & (0.00686) & (0.00141)& (0.0000403)\\[1ex]		
			
			\hline\hline	 		
	\end{tabular}} 

	\label{tb2} 
\end{table}  

\begin{table}[h!]
	\centering 
	\caption {The bias and MSE (in bracket) for the kernel estimator of residual GWIGF in (\ref{eq6.2}) based on the relief times of $20$ patients.}
	\scalebox{1.0}{\begin{tabular}{c c c c c c c c } 
			\hline\hline\vspace{.1cm} 
			$t$ &$ \beta=1.2$ & $\beta=1.4$ &$\beta=1.7$ &$\beta=2.0$ &$\beta=3.0$\\
			\hline\hline
			
			0.1 & -0.10032 & -0.18902 & -0.29483& -0.36566 & -0.46920 \\
			~ &  (0.01783)& (0.03923) & (0.08981) & (0.13722)&  (0.22358)\\[1.2ex]
			0.3 & -0.10171 & -0.19139 & -0.29633& -0.37375 & -0.46926 \\
			~ &  (0.01754)& (0.04031) & (0.09092) & (0.14314)& (0.22289)\\[1.2ex]
			0.6 & -0.10825 & -0.19877 & -0.29892& -0.36767& -0.47252\\
			~ &(0.01937) & (0.04327) & (0.09230) & (0.13858)& ( 0.22604)\\[1.2ex]
			0.8 & -0.13174 & -0.20767 & -0.30882& -0.37731&  -0.47297\\
			~ & (0.02497) & (0.04765) & (0.09863) & (0.14592)& (0.22665)\\[1.2ex]	
			1.0 & -0.16784 & -0.24750 &  -0.32998& -0.39330& -0.48023\\
			~ & (0.03929) & (0.06634) & (0.11255) & (0.15852) & (0.23347)\\[1.2ex]
			1.5 & -0.01982 & -0.11685 &  -0.25049& -0.38708& -0.73248 \\
			~ & (0.05828) & (0.06610) & (0.13428) & (0.21951)& (0.66738)\\[1ex]		
			
			\hline\hline	 		
	\end{tabular}} 
	
	\label{tb3} 
\end{table}

{The GWIGF of exponential distribution with scale parameter $\lambda~(>0)$ is given by \begin{align}\label{eq6.3}
I^w_\beta(X;t)=\beta^2 \lambda^{(\beta-2)}(\beta \lambda t+1),~\beta\ge1.
\end{align}
For estimating (\ref{eq6.3}), we first estimate the unknown model parameter $\lambda$ using maximum likelihood approach. The maximum likelihood estimator (MLE) of $I^w_\beta(X;t)$ is given by 
\begin{align}\label{eq6.4}
\tilde{I}^w_\beta(X;t)=\beta^2 \widehat{\lambda}^{(\beta-2)}(\beta \widehat{\lambda} t+1),
\end{align}
where $\widehat{\lambda}$ is the MLE of $\lambda$.
To compute the bias and MSE  values of the parametric estimator given by (\ref{eq6.4}), we conduct a Monte Carlo simulation using R software with $250$ replications. Here we have considered $\lambda=0.5$. The sample sizes are taken as $n=30, 50,70$ and $100$. Different values of $\beta$ and $t$ are also considered. The MSE values are presented within the first bracket in Table $8$ in each row. From Table $8$, we notice that for fixed values of $\beta$ and $t$, the bias and MSE decrease as $n$ increases. From Table $1$ and Table $8$, we observe that the parametric approach has superior performance than the non-parametric approach in the sense of the absolute bias and MSE values.   }

\begin{table}[h!]
		\caption {{The bias and MSE (in bracket) for the parametric estimator of residual GWIGF in (\ref{eq6.2}).}}
	\centering 
	\scalebox{1.1}{\begin{tabular}{c c c c c c c c } 
	
			\hline\hline\vspace{.1cm} 
			$\beta$ &$t$ & $n=30$ &$n=50$ &$n=70$ &$n=100$\\
			
			\hline\hline
			
			\multirow{10}{1.4cm}{1.2} & 0.1 & -0.01192513 & -0.008816 &-0.010091& -0.005745 \\
			~ & ~ &  (0.03010298)& (0.018208) & (0.0124275) & (0.0083792)  \\[1.2ex]
			~ & 0.2 & -0.011493 &  -0.008534 & -0.009835 &-0.005589   \\
			~ & ~ &  (0.0291723)& (0.017654) & (0.0120484) & ( 0.0081238) \\[1.2ex]
			~ & 0.5 & -0.010195 & -0.007685 & -0.009069&  -0.00512  \\
			~ & ~ &( 0.0264701) & (0.016044) & (0.010947) & (0.0073812)  \\[1.2ex]
			~ & 0.7 & -0.00933 & -0.007119 & -0.008559 & -0.004809 \\
			~ & ~ & ( 0.0247435) & (0.0150142) & ( 0.0102423) & (0.0069061) \\[1.2ex]	
			~ & 0.9 &-0.008465& -0.006553& -0.008048& -0.004497  \\
			~ & ~ & (0.0230766) & ( 0.0140191) & (0.0095613) & (0.0064469) \\[1ex]		
			\hline
			\multirow{10}{1.4cm}{1.7} & 0.1 & -0.001724 &  -0.001218& -0.00127& -0.000743  \\
			~ & ~ &  (0.0003474)& (0.0002095) & (0.0001436) & (0.000097)  \\[1.2ex]
			~ & 0.2 & -0.000741 & -0.000592 & -0.000735 & -0.000411   \\
			~ & ~ &  (0.0001904)& (0.0001169) & (0.0000802) & (0.0000543) \\[1.2ex]
			~ & 0.5 & 0.002207 & 0.001284 & 0.000872 & 0.000587\\
			~ & ~ &(0.0000140) & (0.0000048) & (0.0000022) & (0.0000009)  \\[1.2ex]
			~ & 0.7 & 0.004173 & 0.002535&  0.001943 & 0.001251 \\
			~ & ~ & (0.0001419) & (0.0000679) & (0.0000435) & (0.0000274) \\[1.2ex]	
			~ & 0.9 &0.006138 & 0.003786&  0.003013& 0.001916 \\
			~ & ~ & (0.0004663) & (0.0002413) & (0.0001595) & (0.0001035) \\[1ex]		
			\hline
			\multirow{10}{1.4cm}{2.5} & 0.1 &0.003135 & 0.001983 &  0.001686&  0.001048 \\
			~ & ~ &  (0.0002359)& (0.0001313) & (0.0000886) & (0.0000588)  \\[1.2ex]
			~ & 0.2 &0.004283 & 0.002689 & 0.002257 & 0.001408  \\
			~ & ~ &  (0.0003941)& (0.0002163) & (0.0001454) & (0.0000961) \\[1.2ex]
			~ & 0.5 & 0.007725&  0.004805 & 0.003971 & 0.002488\\
			~ & ~ &(0.0011126) & (0.0005986) & (0.0004000) & (0.0002628)  \\[1.2ex]
			~ & 0.7 & 0.01002 &  0.006216 & 0.005113 & 0.003208 \\
			~ & ~ & (0.0017949) & (0.0009595) & (0.00064) & (0.0004196) \\[1.2ex]	
			~ & 0.9 & 0.012315 & 0.007627 &  0.006255 & 0.003928 \\
			~ & ~ & (0.0026399) & ( 0.0014052) & ( 0.0009361) & (0.0006129) \\[1ex]		
			\hline\hline	 		
	\end{tabular}} 
	\label{tb1} 
\end{table}

\section{Conclusions} 
{Motivated by the information generating function (\cite{golomb1966}) and the weighted Shannon entropy (\cite{guiacsu1971weighted}), here we proposed GWIGF and studied its several properties. It is shown that the GWIGF produces the weighted Shannon entropy, weighted informational energy and other useful weighted information measures. Further, we established that the proposed weighted IGF is shift dependent. This property allows GWIGF to be useful in many areas such as reliability analysis, biology and information theory. Further, it has been established that there is a connection between GWIGF and weighted Shannon entropy. Bounds are useful when it is difficult to obtain explicit form of an informational measure. Here, we have also obtained lower and upper bounds of the GWIGF. Monotone transformations are used in diverse areas of research, such as finance, signal processing and machine learning. We have studied the effect of GWIGF under monotone transformations. Further, we propose a relation between the dispersive order and GWIGF-based order. A connection between the weighted varentropy and varentropy with the newly proposed GWIGF has been established. We have also obtained an upper bound of the sum of two independent random variables using the concept of convolution of two independent random variables. Further, we studied the concept of the GWRIGF. It is shown that the weighted Kullback-Leibler divergence can be generated from the proposed  GWRIGF. The effect of the GWRIGF under strictly monotone transformations has been provided. The GWIGF has been studied for escort, generalized escort and mixture distributions. A close connection between the weighted $\beta$-cross informational energy with the proposed GWIGF has been established for the escort distribution. 

We have also studied GWIGF and GWRIGF for the residual lifetime, which is an important concept in reliability and life testing studies. We obtained a connection between the GWIGF of residual lifetime and cumulative hazard rate. Further, we obtain the GWIGF of residual lifetime of equilibrium distribution in terms of the MRL and weighted MRL. The hazard rate order-based sufficient conditions have been derived under which the residual GWIGFs of two distributions are ordered. The GWRIGF of residual lifetime has been proposed. Its effects under monotone transformations is presented. A non-parametric estimator has been proposed, and its biases and MSEs are computed based on a simulated data set and two real-life data sets. For real data sets, we have used goodness of fit test to obtain the best fitted model. Further, we compared the non-parametric approach with a parametric method. For parametric method, we considered exponential distribution. It is observed that the parametric approach has superior performance than the non-parametric approach in terms of the absolute bias and MSE values. }

\section*{Acknowledgements}   {The authors thank the Editor in Chief, Associate Editor and referees for all their helpful comments and suggestions, which led to the substantial improvements.}  The author Shital Saha thanks the UGC (Award No. $191620139416$), India, for the financial assistantship.

\section*{Conflicts of interest} Both authors declare no conflict of interest.

%
%
%

\bibliography{references}
\end{document}